\documentclass[12pt]{amsart}
\usepackage{setspace}
\usepackage{amsmath, amsfonts, amsthm, amstext, xspace}
\usepackage{amssymb,latexsym}
\usepackage{luatex85}

\usepackage{xcolor}
\definecolor{seagreen}{RGB}{46,139,87}
\definecolor{maroon}{RGB}{128,0,0}
\definecolor{darkviolet}{RGB}{148,0,211}
\definecolor{twelve}{RGB}{100,100,170}
\definecolor{thirteen}{RGB}{100,150,50}
\definecolor{fourteen}{RGB}{200,0,0}
\definecolor{fifteen}{RGB}{0,200,0}
\definecolor{sixteen}{RGB}{0,0,200}
\definecolor{seventeen}{RGB}{200,0,200}
\definecolor{eighteen}{RGB}{0,200,200}

\usepackage[all]{xy}
\newtheorem{thm}{Theorem}[section]
\newtheorem{lem}[thm]{Lemma}
\newtheorem{cor}[thm]{Corollary}

\newtheorem{prop}[thm]{Proposition}

\newtheorem{rem2}[thm]{Remark}
\newtheorem{exm}[thm]{Example}
\newtheorem{fig}[thm]{Figure}

\newtheorem{conv}[thm]{Convention}
\theoremstyle{definition}
\newtheorem{defin}[thm]{Definition}

\def\a{\mathbb{A}}
\def\c{\mathbb{C}}
\def\f{\mathbb{F}}
\def\g{\mathbb{G}}

\def\m{\mathbb{M}}

\def\p{\mathbb{P}}

\def\z{\mathbb{Z}}

\def\cd{\mathcal{D}}

\def\Ext{\operatorname{Ext}}

\usepackage{amssymb}
\usepackage{tikz-cd}

\usepackage{pdfpages}

\author{J.D. Quigley}
\title{The motivic Mahowald invariant}

\begin{document}
\maketitle

\section*{Abstract}
The classical Mahowald invariant is a method for producing nonzero classes in the stable homotopy groups of spheres from classes in lower stems. We study the Mahowald invariant in the setting of motivic stable homotopy theory over $Spec(\c)$. We compute a motivic version of the $C_2$-Tate construction for various motivic spectra, and show that this construction produces ``blueshift" in these cases. We use these computations to show that the Mahowald invariant of $\eta^i$, $i \geq 1$, is the first element in Adams filtration $i$ of the $w_1$-periodic families constructed by Andrews ~\cite{And14}. This provides an exotic periodic analog of Mahowald and Ravenel's computation ~\cite{MR93} that the classical Mahowald invariant of $2^i$, $i \geq 1$, is the first element in Adams filtration $i$ of the $v_1$-periodic families constructed by Adams ~\cite{Ada66}. 

\section{Introduction}

The classical Mahowald invariant is a method for producing nonzero classes in the stable homotopy groups of spheres from classes in lower stems. The classical Mahowald invariant is defined using Lin's Theorem ~\cite{LDMA80}, which says that after $2$-completion, there is an equivalence of spectra
$$S^{0} \simeq \lim_{\underset{n}{\longleftarrow}} \Sigma RP^\infty_{-n}$$
between the sphere spectrum and a homotopy limit of stunted real projective spectra. Let $\alpha \in \pi_t(S^0)_{(2)}$ be a class in the $2$-primary stable stems. Then by the above equivalence, there is some minimal $N>0$ so that the map 
$$S^t \to S^0 \to \Sigma RP^\infty_{-N}$$
is essential. This gives rise to a nontrivial map $(S^t \to S^{-N+1}) \in \pi_{t+N-1}(S^0)$, where $S^{-N+1}$ is the fiber of the collapse map $\Sigma RP^\infty_{-N} \to \Sigma RP^\infty_{-(N-1)}$. The nontrivial map above is the classical Mahowald invariant of $\alpha$, which we denote $M^{cl}(\alpha)$. 

In this paper, we define and analyze an analog of the Mahowald invariant in the setting of motivic stable homotopy theory over $Spec(\c)$. We focus on the interaction between the Mahowald invariant and chromatic homotopy theory. More precisely, we are interested in the interaction between the Mahowald invariant and periodic families in the stable stems. This has been studied extensively in the classical setting. In ~\cite{MS87}, Mahowald and Shick defined a chromatic filtration on the Adams spectral sequence $E_2$-page, and in ~\cite{Shi87}, Shick showed that an algebraic version (with input and output in the $E_2$-page of the Adams spectral sequence) of the classical Mahowald invariant took $v_n$-periodic classes to $v_n$-torsion classes. This is the algebraic form of a conjecture of Mahowald and Ravenel that, roughly speaking, the Mahowald invariant of a $v_n$-periodic class is $v_n$-torsion ~\cite[Conjecture 12]{MR87}.

The previous conjecture has been verified for several cases. Mahowald and Ravenel  showed in ~\cite{MR93} that the classical Mahowald invariant of $2^i$, $i \geq 1$, is the first element in the stable stems in Adams filtration $i$. Subsequent work of Sadofsky ~\cite{Sad92} showed that for $p>3$ one has $\beta_k \in M^{cl}(\alpha_k)$. Computations of Behrens at $p=3$ ~\cite{Beh06} and later at $p=2$ ~\cite{Beh07} provided further evidence that the classical Mahowald invariant of a $v_n$-periodic class in the stable stems is $v_{n+1}$-periodic. In this paper, we will produce motivic analogs of Mahowald and Ravenel's computations of the Mahowald invariant of $2^i$ for $i \geq 1$.

To further explain our goal, we must explain some background on periodicity in classical and motivic stable homotopy theory. By ``periodic family," we will mean a family of elements produced from iterating a non-nilpotent self-map on a finite complex. Classically, this method was first used by Adams in ~\cite{Ada66}, where he used the non-nilpotent self-map $v^4_1 : \Sigma^8 V(0) \to V(0)$ of the mod $2$ Moore spectrum to produce $v_1$-periodic families inside of the image of $J$. The iterated self-map construction was studied at higher heights by Smith ~\cite{Smi70}, Toda ~\cite{Tod71}, and Miller-Ravenel-Wilson  ~\cite{MRW77}. Much has been written about self-maps of finite complexes; see for example the work of Hopkins-Smith ~\cite{HS98}, and the recent work of Behrens-Hill-Hopkins-Mahowald ~\cite{BHHM08} and Bhattacharya-Egger ~\cite{BE16} at $p=2$. 

In the motivic setting over a field of characteristic zero, Levine showed that the classical stable stems sit inside of the motivic stable stems ~\cite{Lev14}. Therefore, all classical periodic families also exist motivically. However, there are non-classical classes in the motivic stable stems, some of which form``exotic" periodic families. The first instance of exotic periodicity is the non-nilpotence of the algebraic Hopf invariant one element $\eta_{alg} \in \pi_{1,1}(S^{0,0})$, which was proven by Morel in ~\cite{Mor12}. In analogy with the work of Adams, one can ask if the cofiber of $\eta$ admits a non-nilpotent self-map. In ~\cite{And14}, Andrews showed that there is a non-nilpotent self-map $w^4_1 : \Sigma^{20,12} C \eta \to C\eta$, and he used this to produce $w_1$-periodic families in the motivic stable stems. This suggests that $\eta$ deserves to be called $w_0$, and that there should be $w_n$-periodic families for all $n$. This particular form of exotic periodicity has been studied further in the work of Gheorghe ~\cite{Ghe17b} at $p=2$ and forthcoming work of Krause ~\cite{Kra18} at all primes.	

We can now state our main result. Reinterpreted, Mahowald and Ravenel's computation says that the classical Mahowald invariant of the $v_0$-periodic class $2^i$ is the first $v_1$-periodic class in Adams filtration $i$. In addition to proving an analogous result in the motivic setting, we prove an exotic analog. Precisely, we show that the motivic Mahowald invariant of the $w_0$-periodic class $\eta^i$ is the first $w_1$-periodic class in Adams filtration $i$. 

\subsection{Outline.} Our starting point is the motivic analog of Lin's Theorem proven by Gregersen in ~\cite{Gre12}, where he constructed a motivic analog of $RP^\infty_{-\infty}$ which exhibits the motivic sphere spectrum as a homotopy limit of motivic stunted projective spectra. In Section 2, we begin by recalling the necessary background from Gregersen's thesis. We then define the \emph{motivic Mahowald invariant} in analogy with the classical Mahowald invariant and compute the motivic Mahowald invariants of the algebraic Hopf invariant one elements of ~\cite{DI13}. We then define the \emph{motivic $C_2$-Tate construction} of a motivic ring spectrum $E$, denoted $E^{tC_2}$, and use this to define some approximations to the motivic Mahowald invariant.

In Section 3, we compute the motivic $C_2$-Tate construction of several motivic ring spectra by comparing the Atiyah-Hirzebruch spectral sequence and the motivic Adams spectral sequence. In particular, we compute the motivic $C_2$-Tate construction for two motivic analogs of $ko$; the analogous classical computations are due to Davis-Mahowald ~\cite{DM84}. Recall the motivic ring spectrum $kq$ constructed by Isaksen-Shkembi in ~\cite{IS11} with motivic cohomology $H^{**}(kq) \cong A//A(1)$. We show that there is an isomorphism in motivic homotopy groups
$$\pi_{**}(kq^{tC_2}) \simeq \lim_{\underset{n}{\longleftarrow}} \bigoplus_{i \geq -n} \pi_{**}(\Sigma^{4i,2i} H\z_2),$$
where $H\z$ is the motivic Eilenberg-MacLane spectrum of the integers with mod two motivic cohomology $H^{**}(H\z) \cong A//A(0)$. This computation is used in later sections to compute the motivic Mahowald invariant of $2^i$.

In order to compute the motivic Mahowald invariant of $\eta^i$, we construct a motivic ring spectrum that detects $w_1$-periodic elements. More precisely, we need a spectrum whose Hurewicz image contains some of the $w_1$-periodic families constructed by Andrews. This spectrum is produced using the forthcoming work of Gheorghe-Wang-Xu ~\cite{GWXPP}, where roughly speaking, they prove that there is an equivalence of categories between $C\tau$-modules and $BP_*BP$-comodules. Here, $C\tau$ is the cofiber of $\tau \in \pi_{0,-1}(S^{0,0})$ studied by Gheorghe in ~\cite{Ghe17}. We produce a $C\tau$-module called $wko$ by writing down the corresponding $BP_*BP$-comodule. This $C\tau$-module has the property that $\overline{H}^{**}(wko) \cong \overline{A}//\overline{A}(1)$, where $\overline{(-)}$ indicates that we are in the $C\tau$-linear setting discussed in ~\cite[Section 5]{Ghe17}. We show that there is an isomorphism in motivic homotopy groups
$$\pi_{**}(wko^{tC_2}) \cong \lim_{\underset{n}{\longleftarrow}} \bigoplus_{i \geq -n} \left( \Sigma^{4i,2i} \pi_{**}(wBP\langle 0\rangle) \oplus \Sigma^{4i-1,2i-1} \pi_{**}(wBP\langle 0\rangle) \right),$$
where $wBP\langle 0 \rangle$ is a spectrum constructed by Gheorghe in ~\cite{Ghe17b} satisfying $\pi_{**}(wBP\langle 0 \rangle) \cong \f_2[w_0]$. 

In Section 4, we use the motivic $C_2$-Tate construction computations from the previous section to compute approximations of the motivic Mahowald invariant based on $kq$ and $wko$. These computations are analogous to ~\cite[Theorem 2.16]{MR93}, where they compute these approximations in the classical setting using $ko$. 

In Section 5, we lift the $kq$- and $wko$-based approximations to full computations of the motivic Mahowald invariant of $2^i$ and $\eta^i$ for all $i \geq 1$. For values of $i$ where the motivic Mahowald invariant lands in the Hurewicz image for $kq$ or $wko$, this lifting is trivial in view of Proposition ~\ref{lift}. For the remaining values of $i$, we pass through a series of approximations obtained by varying the cohomology theory and the filtration of the motivic analog of $RP^\infty_{-N}$. This passage proceeds by induction on $i$, with the induction step completed by comparing Adams spectral sequences in the classical, motivic, and $C\tau$-linear settings. In particular, we use Adams' identification of the $v^4_1$-periodicity operator as a Massey product ~\cite{Ada66b} along with a theorem of Isaksen ~\cite[Theorem 2.1.12]{Isa14} to compare the relevant periodic families at the level of Adams $E_2$-pages. 

\subsection{Acknowledgements.} The author thanks Mark Behrens for his guidance throughout this project and for careful readings of several drafts. The author also thanks Bogdan Gheorghe, Prasit Bhattacharya, Jens Jakob Kjaer, and Jonas Irgens Kylling for helpful discussions. The author was partially supported by NSF grant DMS-1547292.

\section{Definition and elementary computations}

\subsection{Background from motivic homotopy theory}

We begin by recalling some work of Voevodsky ~\cite{Voe03}, Morel-Voevodsky ~\cite{MV99}, and Gregersen ~\cite{Gre12}. We will work in the category of motivic spaces or spectra over $Spec(\c)$ at the prime $p=2$, with everything implicitly completed at $2$. Our goal is to sketch the proof of a motivic analog of Lin's Theorem (Theorem ~\ref{lin} below) which will be used to define the motivic Mahowald invariant. 

Classically, Lin's Theorem is proven by exhibiting an isomorphism between the continuous cohomology of $RP^\infty_{-\infty} = {{\underset{\underset{n}{\longleftarrow}}{\lim}}} RP^\infty_{-n}$ and the \emph{Singer construction} of $\f_2$. By continuous cohomology, we mean the colimit of the cohomology groups $H^*(RP^\infty_{-n})$, which in general does not agree with the cohomology of the homotopy inverse limit $RP^\infty_{-\infty}$. The essential property of the Singer construction is that it associates to a module over the Steenrod algebra $A$ an $Ext_A$-equivalent module. In the proof of Lin's Theorem, this implies that the inverse limit Adams spectral sequence for $RP^\infty_{-\infty}$ is isomorphic to the Adams spectral sequence for $S^{-1}$ from the $E_2$-page onwards. The definition of the classical Singer construction for $\f_2$ can be found in ~\cite{LDMA80} and can be found for more general $A$-modules in ~\cite{AGM85}, where the properties listed above are also proven. 

In ~\cite[Section 3.3]{Gre12}, Gregersen defines the \emph{motivic Singer construction} $R_+(-)$ and proves that it associates an $Ext_A$-equivalent module $R_+(M)$ to any module $M$ over the motivic Steenrod algebra. In order to prove a motivic version of Lin's Theorem, Gregersen constructs a motivic analog of $RP^\infty_{-\infty}$ and shows that its continuous motivic cohomology is isomorphic to the motivic Singer construction of the motivic cohomology of a desuspension of the motivic sphere spectrum. We now recall this construction.

Motivic cohomology with mod $2$ coefficients is represented by the motivic Eilenberg-Mac Lane spectrum for $\f_2$. We denote this spectrum by $H$, and we will denote the motivic cohomology of a point $H^{**}(Spec(k))$ by $\m_2$. When $k=\c$, we have
$$\m_2 \cong \f_2[\tau]$$
where $|\tau| = (0,1)$. 

Let $\g_m$ denote the multiplicative group scheme, $\a^n$ the affine space of rank $n$, and $\p^n$ the projective space of rank $n$. Then there is an equivalence $\p^n \simeq (\a^{n+1} \setminus 0) / \g_m$. Let $\mu_p$ denote the group scheme of $p$-th roots of unity. The closed inclusion $\mu_p \hookrightarrow \g_m$ defines an action of $\mu_p$ on $\a^n \setminus 0$. The motivic lens space is then defined as
$$L^n := (\a^n \setminus 0)/\mu_p.$$

The inclusion $\a^n \setminus 0 \hookrightarrow \a^{n+1} \setminus 0$ sending $(x_1,\ldots,x_n) \mapsto (x_1,\ldots,x_n,0)$ induces a map $L^n \to L^{n+1}$. Taking the colimit over these maps defines a motivic space called $L^\infty$, also known as the geometric classifying space $B(\mu_p)_{gm}$ of the $p$-th roots of unity. 

The motivic spaces $L^n$ are represented by smooth schemes ~\cite[Lemma 4.1.2]{Gre12}, so in particular any algebraic bundle $E \to L^n$ is an $\a^1$-homotopy equivalence by ~\cite[Proposition 4.2.3]{MV99}. The tautological line bundle $\gamma^1_n$ over $\p^n$ can be viewed as the closed subset of $\a^{n+1} \times \p^n$ satisfying 
$$x_i y_j = x_j y_i,$$
where the $x_i$'s are coordinates for $\a^{n+1}$ and the $y_i$'s are coordinates for $\p^n$. 
The inclusions
$$\iota : \p^n \hookrightarrow \p^{n+1}$$
satisfy $\iota^* \gamma^1_{n+1} = \gamma^1_n$. By ~\cite[Lemma 6.3]{Voe03}, there is an identification of $L^n$
$$L^n \cong E((\gamma^1_{n-1})^{\otimes p} \setminus 0 \downarrow \p^{n-1})$$
with the total space of the complement of the zero section of the $p$-fold tensor product of the tautological line bundle over $\p^{n-1}$. The projection thus induces a map
$$f_n : L^n \to \p^{n-1}.$$
The tautological line bundle $\gamma^1_n$ over $L^n$ is the same bundle as the pullback $f^*_n \gamma^1_{n-1}$ of the tautological line bundle over $\p^{n-1}$. 

The motivic cohomology of these spaces was computed by Voevodsky in ~\cite{Voe03}. Below we specialize the computation to the case $k=\c$:

\begin{thm} ~\cite[Theorems 4.1 and 6.10]{Voe03} For any prime $p$,
$$H^{**}{\p^n} \cong \m_p[v]/(v^{n+1})$$
with $|v| = (2,1)$ and for $p=2$,
$$H^{**}(L^n) \cong \m_2[u,v]/(u^2 + \tau v, v^n)$$
$$H^{**}(L^\infty) \cong \m_2[u,v]/(u^2 + \tau v)$$
where $|u| = (1,1)$ and $|v| = (2,1)$.  The action of the motivic Steenrod algebra on $H^{**}(L^\infty)$ is given by
$$Sq^{2i}(v^k) = {2k\choose2i} v^{k+i}$$
$$Sq^{2i+1}(v^k) = 0$$
$$Sq^{2i}(uv^k) = {2k\choose2i} uv^{k+i}$$
$$Sq^{2i+1}(uv^k) = {2k\choose2i}v^{k+i+1}.$$
\end{thm}

The pullback bundle $f^*_n \gamma^1_{n-1}$ over $L^n$ can be identified with $(\a^n \setminus 0) \times_{\mu_p} \a^1$ over $L^n$ by identifying $(\lambda x_1, \ldots, \lambda x_n, y) \sim (x_1,\ldots,x_n,\lambda y)$ where $\lambda \in \mu_p$. The inclusion $\a^1 \hookrightarrow \a^n$ defines an embedding of $\gamma^1_{n-1}$ into the trivial bundle $\epsilon^n$ over $L^n$ by sending
$$(x_1,\ldots,x_n,y) \mapsto (x_1,\ldots,x_n,x_1y,\ldots,x_ny) \in L^n \times \a^n.$$ 

In analogy with the classical construction of $RP^\infty_{-\infty}$, one now needs to define the Thom space of the ``motivic orthogonal complement". For $\eta \hookrightarrow \xi$ an inclusion of vector bundles, Gregersen defines
$$Th(\xi,\eta) := \dfrac{E(\xi)}{E(\xi) \setminus E(\eta)}.$$
By ~\cite[Lemma 4.1.22]{Gre12}, if $\eta \oplus \zeta \cong \xi$ is an isomorphism of vector bundles over a smooth scheme, then $Th(\zeta) \to Th(\xi,\eta)$ is an $\a^1$-weak equivalence. With this notion of orthogonal complement, one can define motivic stunted projective spectra:

\begin{defin} ~\cite[Definition 4.1.23]{Gre12} For $n \geq 0$ and $k \geq 0$, let $\underline{L}^{n-k}_{-k}$ be the motivic spectrum 
$$\underline{L}^{n-k}_{-k} = \Sigma^{-2kn,-kn} Th(k\epsilon^n,k\gamma^1_{n-1}).$$
\end{defin}

We will frequently refer to (stable) cells of $\underline{L}^{n-k}_{-k}$. The following lemma justifies this terminology.

\begin{lem} ~\cite[Lemma 4.2.19]{Gre12} The spectra $\underline{L}^{n-k}_{-k}$ are stably cellular. 
\end{lem}

The following lemmata are proven over more general base schemes, but we specialize to the case $k=\c$. 

\begin{lem}~\cite[Lemma 4.1.24]{Gre12} The motivic cohomology of  $\underline{L}^{n-k}_{-k}$ is given as a module over $\m_2$ by
$$H^{**}(\underline{L}^{n-k}_{-k}) = \Sigma^{-2k,-k} \m_2[u,v]/(u^2 + \tau v, v^n)$$
with $|u| = (1,1)$ and $|v| = (2,1)$. 
\end{lem}

In order to define the analog of $RP^\infty_{-\infty}$, we need to take a homotopy limit of these motivic lens spaces. Therefore, one must know that $n$ and $k$ can be varied compatibly. 

\begin{lem}~\cite[Lemma 4.1.30]{Gre12} The diagram 
\[
\begin{tikzcd}
\underline{L}^{n-k}_{-k} \arrow{r} \arrow{d} & \underline{L}^{n-k+1}_{-k}\arrow{d} \\
\underline{L}^{n-k+1}_{-k+1} \arrow{r} & \underline{L}^{n-k+2}_{-k+1}
\end{tikzcd}
\]
commutes. The induced map in motivic cohomology sends elements $v^j \mapsto v^j$ and $uv^j \to uv^j$ if defined and are zero otherwise. 
\end{lem}

By letting $n$ vary, we obtain a spectrum
$$\underline{L}^\infty_{-k} = \underset{{\underset{n}{\longrightarrow}}}{\lim} \underline{L}^{n-k}_{-k}$$
and we note that as in the classical case, $\underline{L}^\infty_0 = \Sigma^\infty L^\infty_+$. Now letting $k$ vary, we have
$$\underline{L}^\infty_{-\infty} =  \underset{{\underset{k}{\longleftarrow}}}{\lim} \underline{L}^\infty_{-k}.$$

\begin{prop} ~\cite[Proposition 4.1.32]{Gre12} As modules over $\m_2$ we have
$$H^{**}(\underline{L}^\infty_{-k}) \cong \Sigma^{-2k,-k} \m_2[u,v]/(u^2 + \tau v)$$
and
$$H^{**}_c(\underline{L}^\infty_{-\infty}) \cong \m_2[u,v,v^{-1}]/(u^2 + \tau v).$$
\end{prop}

One can extend the isomorphism of $\m_2$-modules in the previous proposition to an isomorphism of $A$-modules by continuing the periodic action of the motivic Steenrod operations on $L^\infty$ to the negative cells. For reference, we include a picture of the motivic cohomology of $\underline{L}^\infty_{-\infty}$ in a range below.

\begin{fig}\label{lcd}
The following diagram depicts the continuous motivic cohomology $H_c^{i,j}(\Sigma^{1,0}\underline{L}^{\infty}_{-\infty})$ for $-11 \leq i \leq 12$ and $-5 \leq j \leq 6$. The horizontal axis is topological degree and the vertical axis is motivic weight. A bullet represents $\m_2$, i.e. it represents an infinite tower of $\f_2$'s connected by $\tau$-multiplication extending downwards. The action of $Sq^1$ is depicted by black horizontal lines between bullets, the action of $Sq^2$ is depicted by blue curves with horizontal length $2$, and the action of $Sq^4$ is depicted by red curves with horizontal length $4$. 

\hskip.4in \includegraphics[scale=.5]{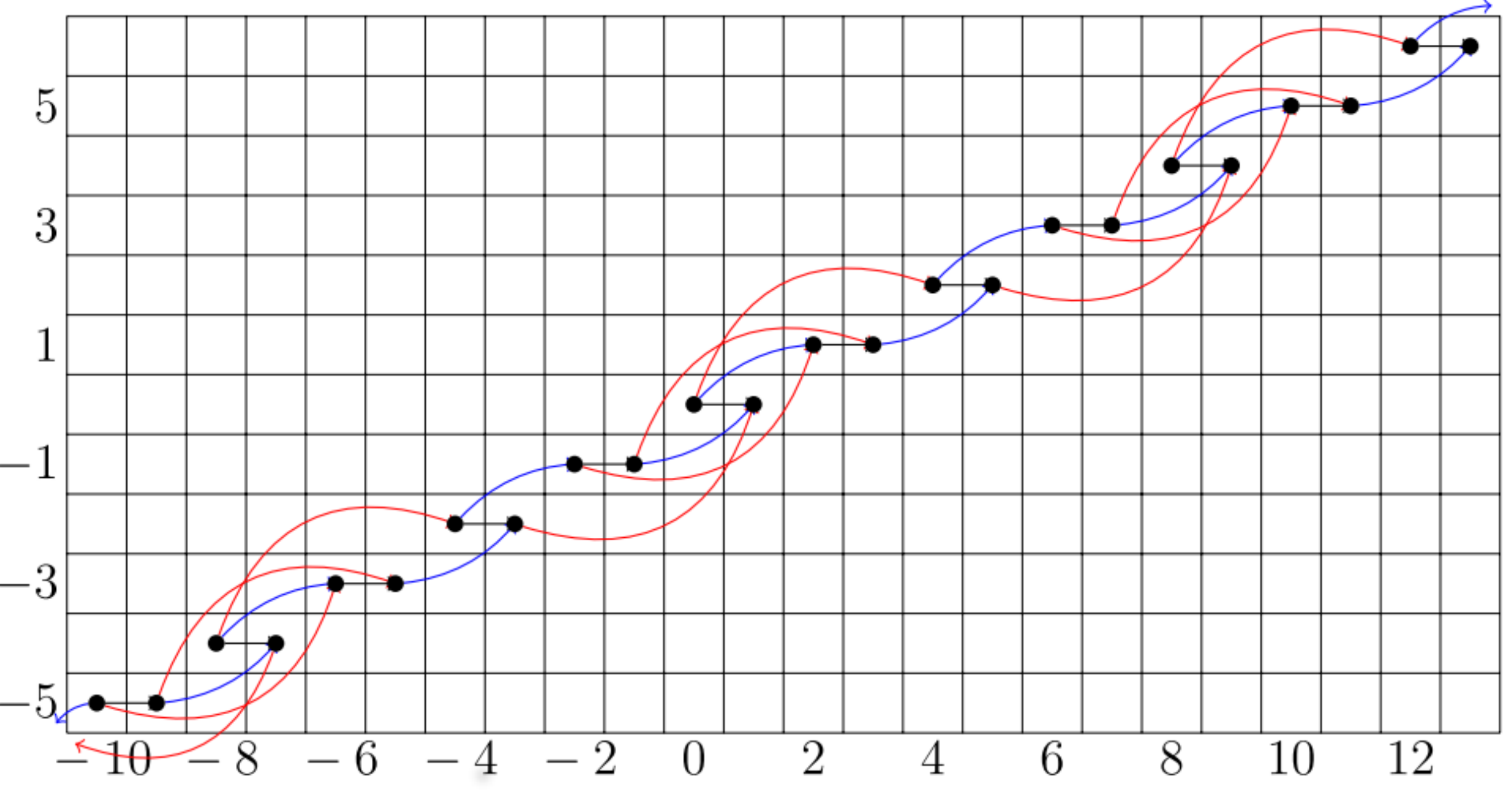}

\end{fig}

The following result implies a motivic analog of Lin's Theorem.

\begin{prop} ~\cite[Proposition 4.1.37]{Gre12} There is an $A$-module isomorphism
$$H^{**}_c(\underline{L}^\infty_{-\infty}) \cong \Sigma^{1,0} \m_2[u,v,v^{-1}]/(u^2 + \tau v) \cong R_+(\m_2),$$
where $R_+(\m_2)$ is the motivic Singer construction of $\m_2$.
\end{prop}

The argument sketched at the beginning of this section using the inverse limit motivic Adams spectral sequence proves the following.

\begin{thm} ~\cite[Theorem 2.0.2]{Gre12}\label{lin} There is a $\pi_{**}$-isomorphism
$$S \to \Sigma^{1,0}\underline{L}^\infty_{-\infty}$$
after $2,\eta$-completion.
\end{thm}

Note that over $Spec(\c)$, $2$-completion and $2,\eta$-completion coincide by ~\cite{HKO11}, so it suffices for us to work in the $2$-completed setting.

\subsection{Definition of the motivic Mahowald invariant}
We begin by recalling the definition of the classical Mahowald invariant.

\begin{defin}~\label{cldef}
Let $\alpha \in \pi_t(S^0)$. The \emph{classical Mahowald invariant} of $\alpha$ is the coset of completions of the diagram
\[
\begin{tikzcd}
S^t \arrow{d}{\alpha} \arrow[r,dashed] & S^{-N+1} \arrow{dd} \\
S^0 \arrow{d}{\simeq} \\
\Sigma RP^\infty_{-\infty} \arrow{r} & \Sigma RP^\infty_{-N}
\end{tikzcd}
\]
where $N>0$ is minimal so that the left-hand composition is nontrivial. The equivalence on the left-hand side is by Lin's Theorem ~\cite{LDMA80}, and the dashed arrow is the lift to the fiber of the sequence
$$S^{-N+1} \to \Sigma RP^\infty_{-N} \overset{c}{\to} \Sigma RP^\infty_{-N+1}$$
which is nontrivial by choice of $N$. The classical Mahowald invariant of $\alpha$ will be denoted $M^{cl}(\alpha)$. 
\end{defin}

We now define the motivic analog of the classical Mahowald invariant. 

\begin{defin}~\label{mmidef} Let $\alpha \in \pi_{s,t}(S^{0,0})$. We define the \emph{motivic Mahowald invariant} of $\alpha$, denoted $M(\alpha)$, as follows. Consider the coset of completions of the following diagram
\[
\begin{tikzcd}
S^{s,t} \arrow[dashed,r] \arrow{d}{\alpha} & S^{-2N+1,-N} \vee S^{-2N+2,-N+1} \arrow{dd} \\
S^{0,0} \arrow{d}{\simeq} \\
\Sigma^{1,0} \underline{L}^\infty_{-\infty} \arrow{r} & \Sigma^{1,0} \underline{L}^\infty_{-N}
\end{tikzcd}
\]
where $N>0$ is minimal so that the left-hand composition is nontrivial. The equivalence on the left-hand side is by Gregersen's Theorem, and the dashed arrow is the lift to the fiber of the sequence
$$S^{-2N+1,-N} \vee S^{-2N+2,-N+1} \to \Sigma^{1,0} \underline{L}^\infty_{-N} \to \Sigma^{1,0} \underline{L}^\infty_{-N+1}$$
which is nontrivial by the choice of $N$. In contrast with the definition of the classical Mahowald invariant, the target of the dashed arrow is a wedge of spheres. If the composition of the dashed arrow with the projection onto the higher dimensional sphere is nontrivial, we define the motivic Mahowald invariant $M(\alpha)$ to be the coset of completions composed with the projection onto the higher dimensional sphere. Otherwise, the composition of the dashed arrow with the projection onto the higher dimensional sphere is trivial and we define the motivic Mahowald invariant $M(\alpha)$ to be the coset of completions composed with the projection onto the lower dimensional sphere. We illustrate this convention in the examples later in this section.
\end{defin}

\begin{rem2}
Note that in the classical setting we could have defined 
$$RP^\infty_{-\infty} = \lim_{\underset{n}{\longleftarrow}} RP^\infty_{-2n}.$$
Then given $\alpha \in \pi_t(S^0)$ we could consider the coset of completions of the following diagram
\[
\begin{tikzcd}
S^t \arrow{d}{\alpha} \arrow[r,dashed] & S^{-2N+1} \vee S^{-2N+2} \arrow{dd} \\
S^0 \arrow{d}{\simeq} \\
\Sigma RP^\infty_{-\infty} \arrow{r} & \Sigma RP^\infty_{-2N}
\end{tikzcd}
\]
where $N>0$ is minimal so that the left-hand composition is nontrivial. If the composition of the dashed arrow with the projection onto the higher dimensional sphere is nontrivial, we can define $\tilde{M}^{cl}(\alpha)$ to be the coset of completions composed with the projection onto the higher dimensional sphere. Otherwise, we can define $\tilde{M}^{cl}(\alpha)$ to be the coset of completions composed with the projection onto the lower dimensional sphere. Then $\tilde{M}^{cl}(\alpha) = M^{cl}(\alpha)$.
\end{rem2}

\subsection{Atiyah-Hirzebruch spectral sequence for $\Sigma^{1,0} \underline{L}^\infty_{-\infty}$}
We can use the Atiyah-Hirzebruch spectral sequence to determine which cell of $\Sigma^{1,0} \underline{L}^\infty_{-\infty}$ a class $\alpha \in \pi_{**}(S^{0,0})$ is detected on. In this subsection, we analyze this spectral sequence in a range. These computations will be used in the next subsection to compute some motivic Mahowald invariants.

The Atiyah-Hirzebruch spectral sequence arises from the cellular filtration of $\Sigma^{1,0} \underline{L}^\infty_{-\infty}$. This spectral sequence has the form
$$E^1_{s,t,u} = \pi_{t,u}(S^{s,\lfloor s/2 \rfloor}) = \pi_{t-s,u-\lfloor s/2 \rfloor}(S^{0,0}) \Rightarrow \pi_{t,u}(\Sigma^{1,0} \underline{L}^\infty_{-\infty}) \cong  \pi_{t,u}(S^{0,0}).$$
We will denote classes in $E^1_{s,t,u}$ by $x[s]$ where $x \in \pi_{**}(S^{0,0})$. 

Differentials are induced by attaching maps in $\Sigma^{1,0} \underline{L}^\infty_{-\infty}$. Let $\eta, \nu,\sigma \in \pi_*(S^0)$ denote the classical Hopf invariant one elements and let $\eta_{alg}, \nu_{alg}, \sigma_{alg} \in \pi_{**}(S^{0,0})$ denote the motivic Hopf invariant one elements constructed in ~\cite{DI13}. Recall the following correspondence between attaching maps and $A$-module structure. 

\begin{lem}~\label{dam}
In $H^*(RP^\infty_{-\infty})$, $Sq^1$ detects a $\cdot 2$-attaching map, $Sq^2$ detects an $\eta$-attaching map, $Sq^4$ detects a $\nu$-attaching map, and $Sq^8$ detects a $\sigma$-attaching map.

In $H^{**}(\underline{L}^\infty_{-\infty})$, $Sq^1$ detects a $\cdot 2$-attaching map, $Sq^2$ detects an $\eta_{alg}$-attaching map, $Sq^4$ detects a $\nu_{alg}$-attaching map, and $Sq^8$ detects a $\sigma_{alg}$-attaching map.
\end{lem}

The classical part of the lemma follows from ~\cite[Lemma I.5.3]{SE62}, and the motivic part follows from ~\cite[Remark 4.14]{DI13}. One can apply the classical part of the lemma to show that $\eta \in M^{cl}(2)$, $\nu \in M^{cl}(\eta)$, and $\sigma \in M^{cl}(\nu)$, recovering a result of Mahowald-Ravenel ~\cite[Proposition 2.3]{MR93}. 

In view of the lemma, the $d^1$-differentials in the Atiyah-Hirzebruch spectral sequence have the form $d^1(x[k])) = 2x[k-1]$ for $k \equiv 1 \mod 2$. Examination of ~\cite[Page 8]{Isa14b} gives the differentials in the following figure.

\begin{fig}
The $E^1$-page of the Atiyah-Hirzebruch spectral sequence for $-5 \leq s \leq 5$ and $-5 \leq t \leq 5$ with $d^1$-differentials included. 

\hskip1.6in \includegraphics[scale=.5]{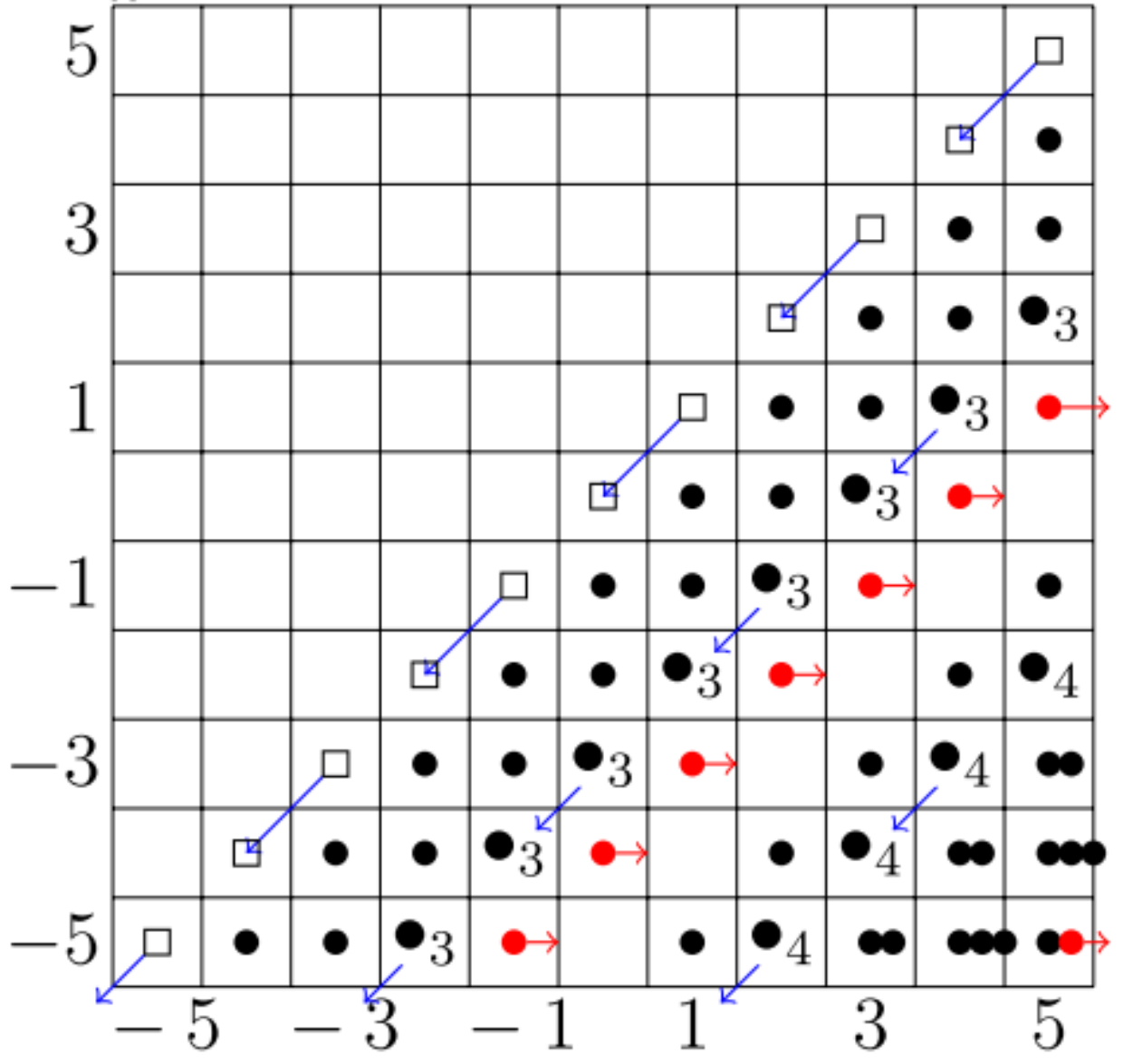}

A $\square$ represents $\z_{(2)}[\tau]$, a $\bullet$ represents $\f_2[\tau]$, a $\bullet_n$ represents $\z/2^n[\tau]$, and a red $\bullet$ represents $\f_2$. Differentials are blue and $\tau$-linear, and in this chart, are multiplication by two.
\end{fig}

The $d^2$-differentials have the form $d^1(x[k])) = \eta_{alg} x[k-2]$ for $k \equiv 1,2 \mod 4$. Examination of ~\cite[Page 8]{Isa14b} gives the differentials in the following figure.

\begin{fig}
The $E^2$-page of the Atiyah-Hirzebruch spectral sequence for $-5 \leq s \leq 5$ and $-5 \leq t \leq 5$ with $d^2$-differentials included.

\hskip1.6in \includegraphics[scale=.5]{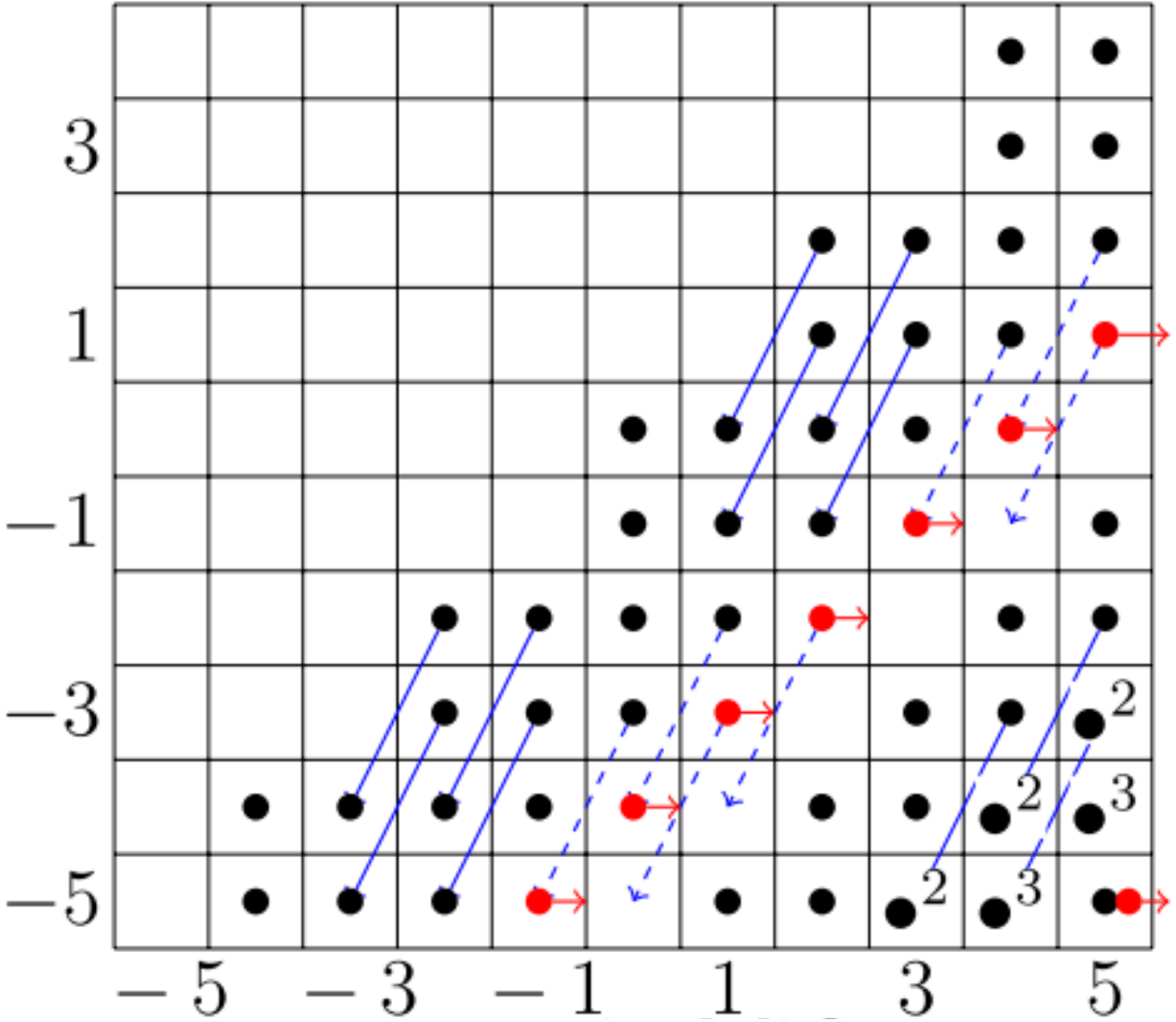}

A $\bullet$ represents $\f_2[\tau]$, a $\bullet^n$ represents $(\f_2[\tau])^{\oplus n}$, and a red $\bullet$ represents $\f_2$. Differentials are blue and $\tau$-linear; a dashed differential means the source, target, or both are $\tau$-torsion.
\end{fig}

The $d^3$-differentials have the form $d^3(x[k]) = \langle x, 2,\eta_{alg} \rangle[k-3]$ for $k \equiv 3 \mod 8$ or $d^3(x[k]) = \langle x, \eta_{alg}, 2 \rangle[k-3]$ for $k \equiv 5 \mod 8$. The $d^4$-differentials have the form $d^4(x[k]) = \nu_{alg}[k-4]$ for $k \equiv 1,2,3,4 \mod 8$. There are no $d^5$-differentials; these correspond to attaching maps $\eta_{alg}^4 : S^{n,m} \to S^{n-4,m-4}$, but every class of the form $\eta_{alg}^4 x[k]$ is the target of a $d^2$-differential. The $d^6$-differentials have the form $d^6(x[k])) = \langle x, \eta_{alg}, \nu_{alg} \rangle[k-6]$ for $k \equiv 6 \mod 16$ or $d^6(x[k]) = \langle x , \nu_{alg}, \eta_{alg} \rangle[k-6]$ for $k \equiv 10 \mod 16$. Examination of ~\cite[Page 8]{Isa14b} and ~\cite[Table 23]{Isa14} gives the differentials in the following figure.

\begin{fig}
The $E^3$-page of the Atiyah-Hirzebruch spectral sequence for $-4 \leq s \leq 5$ and $-5 \leq t \leq 5$ with some of the $d^3$- through $d^6$-differentials included. A $\bullet$ represents $\f_2[\tau]$. Differentials are blue and $\tau$-linear.

\hspace{1.6in}\includegraphics[scale=.5]{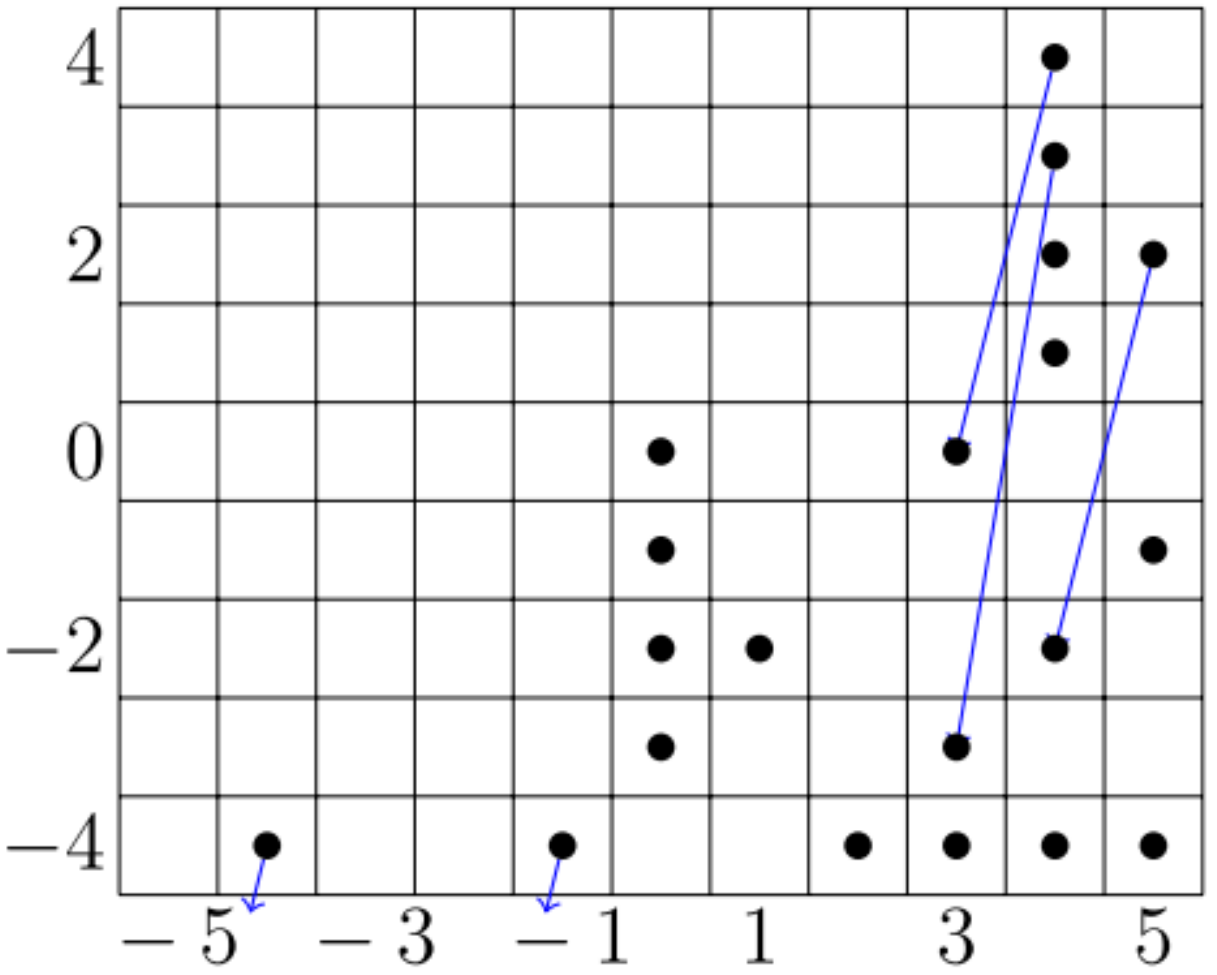}
\end{fig}

\begin{fig}\label{helpme}
The $E^7$-page of the Atiyah-Hirzebruch spectral sequence for $-4 \leq s \leq 0$ and $-4 \leq t \leq 3$. In the lower half of the figure, each $\bullet$ is labeled with the class it detects in $\pi_{**}(S^{0,0})$; we have omitted the subscript `$alg$' for space considerations.

\hspace{1.7in} \includegraphics[scale=.5]{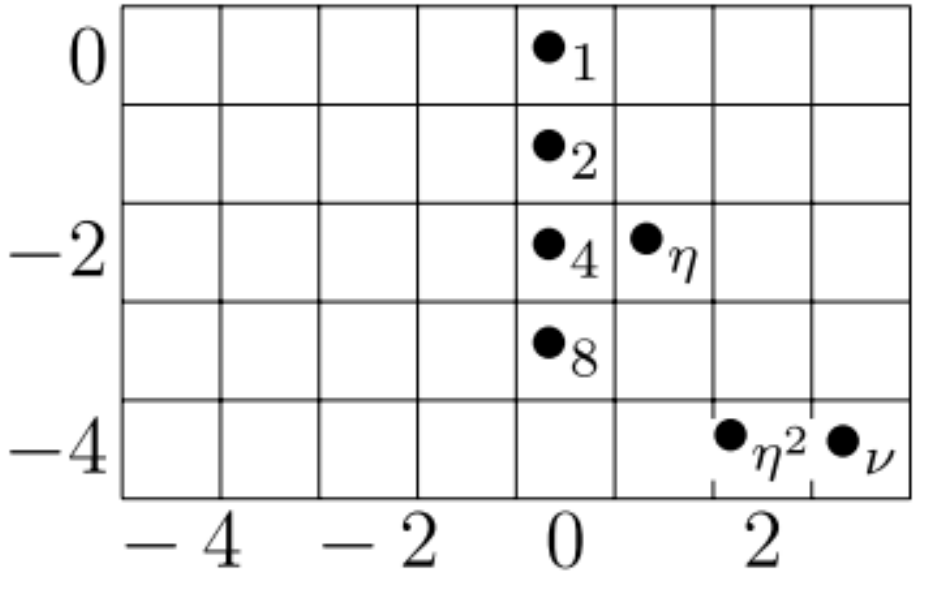}
\end{fig}

There could \emph{a priori} be further differentials in this range. However, we will see in the next subsection that further differentials would produce a contradiction to classical computations of the Atiyah-Hirzebruch spectral sequence for $\Sigma RP^\infty_{-\infty}$ after applying Betti realization. Therefore we have the following.

\begin{prop}
Figure ~\ref{helpme} depicts the $E^\infty$-page of the Atiyah-Hirzebruch spectral sequence for $-4 \leq s \leq 0$ and $-4 \leq t \leq 3$. 
\end{prop}

\subsection{Elementary computations} 
We will now compute the motivic Mahowald invariants of some classes in low degrees. 

\begin{prop} We have $\eta_{alg} \in M(2) \subset \pi_{1,1}(S).$
\end{prop}

\begin{proof}
We must determine the minimal $N >0$ such that the composition
$$S^{0,0} \overset{\cdot 2}{\to} S^{0,0} \simeq \Sigma^{1,0}\underline{L}^\infty_{-\infty} \to \Sigma^{1,0}\underline{L}^\infty_{-N}$$
is nontrivial. We claim that the composition
\begin{align}
S^{0,0} \overset{\cdot 2}{\to} S^{0,0} \simeq \Sigma^{1,0}\underline{L}^\infty_{-\infty} \to \Sigma^{1,0}\underline{L}^\infty_{-1}
\end{align}
is nontrivial. To show this, we use the Betti realization functor of Morel-Voevodsky ~\cite[Section 3.3.2]{MV99}
$$Re_\c : SH_\c \to SH$$
from the motivic stable homotopy category over $Spec(\c)$ to the classical stable homotopy category. Betti realization carries motivic spheres $S^{m,n}$ to classical spheres $S^m$ and therefore induces a map $\pi_{**}(S^{0,0}) \to \pi_*(S^0)$. This can be used to show that $Re_\c(\underline{L}^\infty_{-N}) \simeq RP^\infty_{-2N}$. Therefore applying $Re_\c$ to the composite above gives
$$S^0 \overset{\cdot 2}{\to} S^0 \simeq \Sigma RP^\infty_{-\infty} \to \Sigma RP^\infty_{-2},$$
which is nontrivial by the computation $\eta \in M^{cl}(2)$ from ~\cite[Proposition 2.3]{MR93}. Therefore the composition $(1)$ is nontrivial and we see that $M(2)$ is the coset of completions of the diagram
\[
\begin{tikzcd}
S^{0,0} \arrow[dashed,r] \arrow{d}{\cdot 2} & S^{0,0} \vee S^{-1,-1} \arrow{dd} \\
S^{0,0} \arrow{d}  \\
\Sigma^{1,0} \underline{L}^\infty_{-\infty} \arrow{r} & \Sigma^{1,0} \underline{L}^\infty_{-1}.
\end{tikzcd}
\]
We now must determine which cell in the wedge the motivic Mahowald invariant $M(2)$ is detected on. We claim that composition of the dashed arrow with projection onto the higher dimensional sphere in the wedge $S^{0,0}$ is null. To see this, note that there is no class in Atiyah-Hirzebruch filtration $s=0$ which could detect $2$ by Figure ~\ref{helpme}. Therefore $M(2)$ is detected on $S^{-1,-1}$ so $M(2) \subset \pi_{1,1}(S^{0,0})$. The only nontrivial class in this bidegree is $\eta_{alg}$, so we conclude $\eta_{alg} \in M(2)$.
\end{proof}

\begin{prop}
We have $\nu_{alg} \in M(\eta_{alg}) \subset \pi_{3,2}(S).$
\end{prop}

\begin{proof}
We must determine the minimal $N>0$ such that the composition
$$S^{1,1} \overset{\eta_{alg}}{\to} S^{0,0} \simeq \Sigma^{1,0}\underline{L}^\infty_{-\infty} \to \Sigma^{1,0}\underline{L}^\infty_{-N}$$
is nontrivial. First, we observe that the composition where $N=1$ 
\begin{align}
S^{1,1} \overset{\eta_{alg}}{\to} S^{0,0} \simeq \Sigma^{1,0}\underline{L}^\infty_{-\infty} \to \Sigma^{1,0}\underline{L}^\infty_{-1}
\end{align}
is trivial. This can be seen from Figure ~\ref{helpme} since taking $N=1$ corresponds to removing everything below $s = -1$. In filtrations $s \geq -1$, there are no classes in $E^7$ which could detect $\eta_{alg}$. Therefore the above composition where $N=1$ is trivial. 

Now, we claim that the composition where $N=2$
\begin{align}
S^{1,1} \overset{\eta_{alg}}{\to} S^{0,0} \simeq \Sigma^{1,0}\underline{L}^\infty_{-\infty} \to \Sigma^{1,0}\underline{L}^\infty_{-2}
\end{align}
is nontrivial. Applying Betti realization gives
$$S^1 \overset{\eta}{\to} S^0 \simeq \Sigma RP^\infty_{-\infty} \to \Sigma RP^\infty_{-4}$$
which is nontrivial by the computation $\nu \in M^{cl}(\eta)$ from ~\cite[Proposition 2.3]{MR93}. Therefore the composition $(3)$ is nontrivial and  $M(\eta_{alg})$ is the coset of completions of the diagram
\[
\begin{tikzcd}
S^{1,1} \arrow[dashed,r] \arrow{d}{\eta_{alg}} & S^{-2,-1} \vee S^{-3,-2} \arrow{dd} \\
S^{0,0} \arrow{d}  \\
\Sigma^{1,0} \underline{L}^\infty_{-\infty} \arrow{r} & \Sigma^{1,0} \underline{L}^\infty_{-2}.
\end{tikzcd}
\]
We claim that composition of the dashed arrow with projection onto the higher dimensional sphere in the wedge $S^{-2,-1}$ is nontrivial. To see this, note that Atiyah-Hirzebruch filtration $s=-2$ in Figure ~\ref{helpme} is the only filtration where a class could detect $\eta_{alg}$ without contradicting the upper bound imposed by Betti realization. Therefore $M(\eta_{alg}) \subset \pi_{3,2}(S^{0,0}) \cong \z/8\{ \nu_{alg} \}$. By examination of the Atiyah-Hirzebruch spectral sequence $E^7$-page, we can see that the class detecting $\eta_{alg}$ is $\nu_{alg}[-2]$. Therefore $\nu_{alg} \in M(\eta_{alg})$.
\end{proof}

\begin{rem2}
We cannot use Betti realization alone to compute motivic Mahowald invariants. For example, in the previous proof we were able to use Betti realization to obtain an upper bound on the dimension of the motivic Mahowald invariant, but not a lower bound. Applying Betti realization to the composition $(2)$ gives
$$S^1 \overset{\eta}{\to} S^0 \simeq \Sigma RP^\infty_{-\infty} \to \Sigma RP^\infty_{-2}$$
which is trivial by the computation $\nu \in M^{cl}(\eta)$ from ~\cite[Proposition 2.3]{MR93}. However, Betti realization may send a nontrivial class $x \in \pi_{**}(S^{0,0})$ in the motivic stable stems to a trivial class $Re_\c(x) = 0 \in \pi_{*}(S^0)$ in the classical stable stems. Therefore we cannot conclude triviality of the composition $(2)$ from triviality of its Betti realization. 
\end{rem2}

\begin{prop}
We have $\sigma_{alg} \in M(\nu_{alg}) \subset \pi_{7,4}(S).$
\end{prop}

\begin{proof}
We must determine the minimal $N>0$ such that the composition
$$S^{3,2} \overset{\nu_{alg}}{\to} S^{0,0} \simeq \Sigma^{1,0}\underline{L}^\infty_{-\infty} \to \Sigma^{1,0}\underline{L}^\infty_{-N}$$
is nontrivial. First, we observe that the compositions where $N=1$ or $N=2$
\begin{align}
S^{3,2} \overset{\nu_{alg}}{\to} S^{0,0} \simeq \Sigma^{1,0}\underline{L}^\infty_{-\infty} \to \Sigma^{1,0}\underline{L}^\infty_{-1}
\end{align}
\begin{align}
S^{3,2} \overset{\nu_{alg}}{\to} S^{0,0} \simeq \Sigma^{1,0}\underline{L}^\infty_{-\infty} \to \Sigma^{1,0}\underline{L}^\infty_{-2}
\end{align}
are trivial. This can be seen from Figure ~\ref{helpme} since taking $N=1$ or $N=2$ corresponds to removing everything below $s = -1$ or $s = -3$, respectively. In filtrations $s \geq -3$, there are no classes in $E^7$ which could detect $\nu_{alg}$. Therefore the above compositions where $N=1$ or $N=2$ are trivial. 

Now, we claim that the composition where $N=3$
\begin{align}
S^{3,2} \overset{\nu_{alg}}{\to} S^{0,0} \simeq \Sigma^{1,0}\underline{L}^\infty_{-\infty} \to \Sigma^{1,0}\underline{L}^\infty_{-3}
\end{align}
is nontrivial. Applying Betti realization gives
$$S^3 \overset{\eta}{\to} S^0 \simeq \Sigma RP^\infty_{-\infty} \to \Sigma RP^\infty_{-6}$$
which is nontrivial by the computation $\sigma \in M^{cl}(\nu)$ from ~\cite[Proposition 2.3]{MR93}. Therefore the composition $(6)$ is nontrivial and  $M(\nu_{alg})$ is the coset of completions of the diagram
\[
\begin{tikzcd}
S^{3,2} \arrow[dashed,r] \arrow{d}{\nu_{alg}} & S^{-4,-2} \vee S^{-5,-3} \arrow{dd} \\
S^{0,0} \arrow{d}  \\
\Sigma^{1,0} \underline{L}^\infty_{-\infty} \arrow{r} & \Sigma^{1,0} \underline{L}^\infty_{-3}.
\end{tikzcd}
\]
We claim that composition of the dashed arrow with projection onto the higher dimensional sphere in the wedge $S^{-4,-2}$ is nontrivial. To see this, note that Atiyah-Hirzebruch filtration $s=-4$ is the only filtration where a class could detect $\nu_{alg}$ without contradicting the upper bound given by Betti realization. Therefore $M(\nu_{alg}) \subset \pi_{7,4}(S^{0,0}) \cong \z/16\{ \sigma_{alg} \}$. By examination of the Atiyah-Hirzebruch spectral sequence $E^7$-page, we can see that the class detecting $\nu_{alg}$ is $\sigma_{alg}[-4]$. Therefore $\sigma_{alg} \in M(\nu_{alg})$.
\end{proof}

\subsection{Approximations to the motivic Mahowald invariant}
As we saw above, computing $M(\alpha)$ from the $A$-module structure of $H^{**}(\underline{L}^\infty_{-N})$ and the Atiyah-Hirzebruch spectral sequence  for a class $\alpha \in \pi_{**}(S^{0,0})$ can be somewhat involved. We will use the following approximations to compute the motivic Mahowald invariant. First, we give a definition.

\begin{defin}
The \emph{motivic $C_2$-Tate construction} of a motivic spectrum $E$ is defined by
$$E^{tC_2} := \lim_{\underset{k}{\longleftarrow}} (\Sigma^{1,0} E \wedge \underline{L}^\infty_{-k}).$$ 
\end{defin}

Note that this is \emph{not} the categorical $C_2$-Tate construction in the motivic stable homotopy category, which would use the universal space $EC_2$ over the simplicial classifying space of $C_2$. In contrast with the cells of $\underline{L}^\infty$, the cells of the simplicial classifying space for $C_2$ are concentrated in motivic weight zero. In view of the Atiyah-Hirzebruch spectral sequence computations above and in the sequel, this would not work for our purposes. 

\begin{defin}\label{mie} Let $E$ be a motivic spectrum, and let $\alpha \in \pi_{s,t}(E^{tC_2})$. We define the \emph{motivic $E$-Mahowald invariant}, denoted $M_E(\alpha)$, as follows. Consider the coset of completions of the following diagram
\[
\begin{tikzcd}
S^{s,t} \arrow[r, dashed]\arrow{d}{\alpha} & \Sigma^{-2N+1,-N} E \vee \Sigma^{-2N+2,-N+1}E \arrow{d}\\
E^{tC_2} \arrow{r}& \Sigma^{1,0} E \wedge \underline{L}^\infty_{-N} 
\end{tikzcd}
\]
where $N>0$ is minimal so that the left-hand composition is nontrivial. If the composition of the dashed arrow with the projection onto the more highly suspended wedge summand is nontrivial, we define $M_E(\alpha)$ to be the coset of completions composed with the projection onto this summand. Otherwise, as in Definition ~\ref{mmidef}, we define $M_E(\alpha)$ to be the coset of completions composed with the projection onto the less highly suspended wedge summand.

We will often use splittings $\pi_{**}(E^{tC_2}) \cong \bigoplus_{i \in \z}\pi_{**}(\Sigma^{si,ti}R)$. An element $x \in \pi_{**}(R)$ defines an element $\alpha \in \pi_{**}(E^{tC_2})$ by mapping into the summand with $i=0$, and we will use the notation $M_E(x)$ instead of $M_E(\alpha)$ in this situation.
\end{defin}

Note that if $E=S$, this recovers the motivic Mahowald invariant. The following proposition is a motivic analog of Theorem 2.15 of ~\cite{MR93}. It will allow us to ``lift" motivic $E$-Mahowald invariant computations to motivic Mahowald invariant computations.

\begin{prop}\label{lift}
Let $E$ be a motivic spectrum with a map $S^{0,0} \to E$. If $M_E(x)$ is a coset in $\pi_{**}(E)$ containing a class in the image of a class $y \in \pi_{**}(S^{0,0})$ such that the following diagram commutes,
\[
\begin{tikzcd}
S^{s,t} \arrow{r}{y} \arrow{d}{x} & S^{-2N+1,-N} \vee S^{-2N + 2, -N+1} ,\arrow{d}{i}\\
S^{0,0} \arrow{r}{h} & \Sigma^{1,0} \underline{L}^\infty_{-N},
\end{tikzcd}
\]
then $y \in M(x)$. 
\end{prop}

\begin{proof}
Consider the commutative diagram
\[
\begin{tikzcd}
& & \Sigma^{1,0} \underline{L}^\infty_{-N+1} \arrow{r} & \Sigma^{1,0} \underline{L}^\infty_{1-N} \wedge E \\
S^{s,t} \arrow{r}{x} & S^{0,0} \arrow{r}{h} & \Sigma^{1,0} \underline{L}^\infty_{-N} \arrow{u} \arrow{r}{\eta} & \Sigma^{1,0} \underline{L}^\infty_{-N} \wedge E \arrow{u} \\
& & S^{0,0} \arrow{u} \arrow{r} & \Sigma^{-2N+1,-N} E \vee \Sigma^{-2N+2,-N} E \arrow{u}
\end{tikzcd}
\]
If $\eta \circ h \circ x$ is essential, then $h \circ x$ is essential and $y \in M(x)$ by commutativity of the diagram in the hypothesis.
\end{proof}

The motivic Mahowald invariant is calculated by determining the exact cell of $\underline{L}^\infty_{-\infty}$ where a class $\alpha \in \pi_{**}(S^{0,0})$ is detected. We will occasionally want a ``coarser" invariant which is calculated by determining on which collection of adjacent cells a class is detected. 

\begin{defin}\label{mibx}
Let $\alpha \in \pi_{s,t}(S^{0,0})$ and let $X = \underline{L}^{k}_{k-j}$ for some $k,j$ such that  $k\geq j$. Note that $X$ is a $(j+1)$-cell complex. Suppose that $\underline{L}^k_{k-j} \simeq \Sigma^{-2n(j+1),-n(j+1)} \underline{L}^{(j+1)n+k}_{(j+1)n+k-j}$ for all $n \in \z$. The \emph{motivic Mahowald invariant based on $X$} is the coset of completions of the following diagram
\[
\begin{tikzcd}
S^{s,t} \arrow[r, dashed]\arrow{d}{\alpha} & \Sigma^{1,0} \underline{L}^{k-N(j+1)}_{k-j-N(j+1)} \simeq  \Sigma^{-2N(j+1)+1,-N(j+1)}\underline{L}^k_{k-j} \arrow{dd}  \\
S^{0,0} \arrow{d}{\simeq}  \\
\underset{\underset{n}{\longleftarrow}}{\lim} \Sigma^{1,0} \underline{L}^\infty_{k-j-n(j+1)} \arrow{r}& \Sigma^{1,0} \underline{L}^\infty_{k-j-N(j+1)} 
\end{tikzcd}
\]
where $N>0$ is minimal such that the left-hand composition is nontrivial. The motivic Mahowald invariant based on $X$ of $\alpha$ will be denoted $M(\alpha;X)$.
\end{defin}

The condition $\underline{L}^k_{k-j} \simeq \Sigma^{-2n(j+1),-n(j+1)} \underline{L}^{(j+1)n+k}_{(j+1)n+k-j}$ for all $n \in \z$ says that  $\underline{L}^\infty_{-\infty}$ is ``built from" suspensions of $X$. We illustrate this in the following two examples which will be used in the sequel.

\begin{exm}
Let $k=0$, $j = 1$, and let $X = \underline{L}^0_{-1}$. Then $X$ is equivalent to the $(-1,0)$-desuspension of the mod two Moore spectrum $V(0)$, and by examination of the attaching map structure in $\underline{L}^\infty_{-\infty}$ we have
$$X = \underline{L}^0_{-1}  \simeq \Sigma^{-1,0} V(0) \simeq   \Sigma^{-4n,-2n} \Sigma^{4n-1,2n} V(0) \simeq \Sigma^{-4n,-2n} \underline{L}^{2n}_{2n-1}$$
for each $n \in \z$. In particular, we can define the motivic Mahowald invariant based on (a desuspension of) the mod two Moore spectrum.
\end{exm}

\begin{exm}
Let $k=1$, $j =3$, and let $X = \underline{L}^1_{-2}$. Then $X$ is a $4$-cell complex with motivic cohomology depicted below:

\hskip2in \includegraphics[scale=.2]{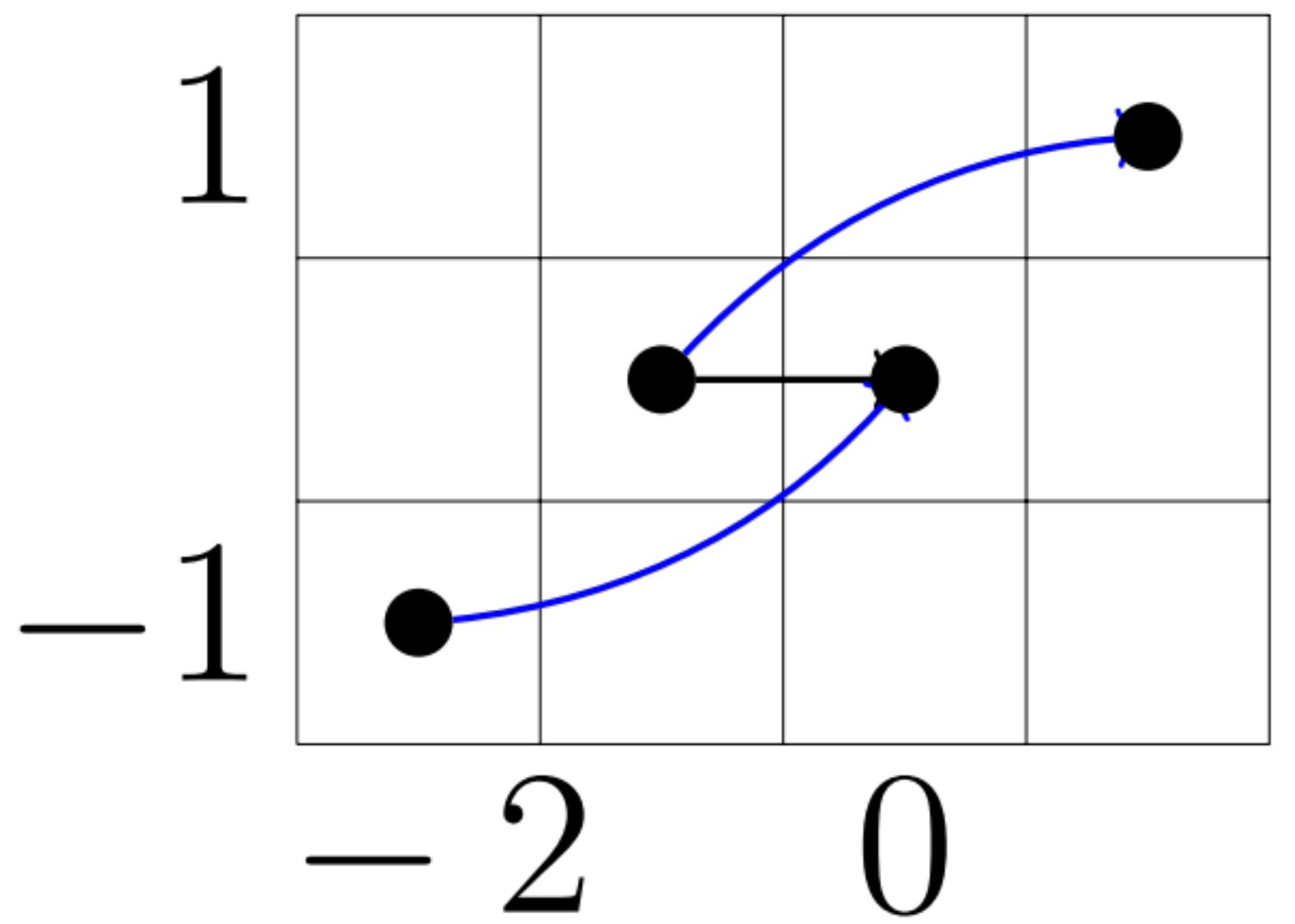}

Examination of Figure ~\ref{lcd} shows that $\underline{L}^\infty_{-\infty}$ can be built from suspensions of $X$. Indeed, we have
$$X = \underline{L}^1_{-2} \simeq \Sigma^{-8n,-4n} \underline{L}^{4n+1}_{4n-2}$$
for each $n \in \z$ since the only attaching maps in $\underline{L}^{4n+1}_{4n-2}$ are the $\cdot 2$- and $\eta$-attaching maps detected by the action of $Sq^1$ and $Sq^2$ on $H^{**}(\underline{L}^{4n+1}_{4n-2})$. The only other possible attaching map is an $\eta^2$-attaching map from the top cell to the bottom cell, but this is not possible since $\eta^2$ is divisible by $\eta$. 
\end{exm}

We can combine the previous two definitions to obtain another approximation.

\begin{defin}\label{miebx} Let $\alpha \in \pi_{s,t}(E^{tC_2})$ and let $X = \underline{L}^k_{k-j}$ as above. The \emph{motivic $E$-Mahowald invariant based on $X$} of $\alpha$ is the coset of completions of the following diagram
\[
\begin{tikzcd}
S^{s,t} \arrow[r, dashed]\arrow{d}{\alpha} & \Sigma^{1,0} \underline{L}^{k-N(j+1)}_{k-j-N(j+1)} \wedge E \simeq  \Sigma^{-2N(j+1)+1,-N(j+1)}\underline{L}^k_{k-j} \wedge E \arrow{dd}  \\
E^{tC_2} \arrow{d}{\simeq}  \\
\underset{\underset{n}{\longleftarrow}}{\lim} \Sigma^{1,0} \underline{L}^\infty_{k-j-n(j+1)} \wedge E \arrow{r}& \Sigma^{1,0} \underline{L}^\infty_{k-j-N(j+1)} \wedge E
\end{tikzcd}
\]
where $N>0$ is minimal such that the left-hand composition is nontrivial. The motivic $E$-Mahowald invariant based on $X$ of $\alpha$ will be denoted $M_E(\alpha;X)$. 
\end{defin}

\begin{rem2}
The approximations defined above are the motivic analogs of ones defined in ~\cite{MR93}. For example, they compute the approximation $M^{cl}_{ko}(2^i; V(0))$ in order to show that $M^{cl}(2^i)$ consists of certain $v_1$-periodic families in the image of $J$. 
\end{rem2}

\section{$C_2$-Tate constructions of some motivic ring spectra}

In order to compute the motivic $E$-Mahowald invariant, we need to understand the $C_2$-Tate constructions of certain motivic ring spectra. The classical result for which we need motivic analogs is the following:

\begin{thm} ~\cite{DM84}
There is an equivalence of spectra
$$ko^{tC_2} \simeq \bigvee_{i \in \z} \Sigma^{4i} H\z.$$
\end{thm}

This splitting is used in ~\cite{MR93} to compute $M^{cl}(2^i)$. To compute the (motivic) $E$-Mahowald invariant, it suffices to produce this splitting at the level of homotopy groups. Before proceeding, we make the following convention to simplify notation.

\begin{conv}
In the sequel, we will write $\eta, \nu, \sigma \in \pi_{**}(S^{0,0})$ to denote the motivic Hopf invariant one elements of ~\cite{DI13} instead of $\eta_{alg}, \nu_{alg}, \sigma_{alg}$. 
\end{conv}

\subsection{Splitting needed for computing $M(2^i)$} In the classical setting, Mahowald and Ravenel use the $ko$-Mahowald invariant to compute $M^{cl}(2^i)$. We will use the motivic analog of $ko$, connective Hermitian K-theory $kq$, to compute $M(2^i)$. The properties of this spectrum we need are due to Isaksen-Shkembi.

\begin{lem} ~\cite{IS11} The motivic cohomology of $kq$ is given by
$$H^{**}(kq) \cong A//A(1)$$
where $A(1)$ is the subalgebra of the motivic Steenrod algebra generated by $Sq^1$ and $Sq^2$. 

The motivic homotopy groups of $kq$ are given by
$$\pi_{**}(kq) \cong \z_2[\tau,\eta,\alpha,\beta]/(2\eta,\tau \eta^3, \eta \alpha, \alpha^2 -4\beta)$$
with $|\tau| = (0,-1)$, $|\eta| = (1,1)$, $|\alpha| = (4,2)$, and $|\beta| = (8,4)$. 
\end{lem}

\begin{prop}~\label{buh}
There is an isomorphism in homotopy groups
$$\pi_{**}(kq^{tC_2}) \simeq \lim_{\underset{n}{\longleftarrow}} \bigoplus_{i \geq -n} \pi_{**}(\Sigma^{4i,2i} H\z_2)$$
\end{prop}
\begin{proof}
Consider the Atiyah-Hirzebruch spectral sequence resulting from the cellular filtration of $\Sigma^{1,0} \underline{L}^\infty_{-\infty}$. This spectral sequence has the form
$$E^1_{s,t,u} = kq_{t,u}(S^{s,\lfloor s/2\rfloor}) = \pi_{t-s,u-\lfloor s/2 \rfloor}(kq) \Rightarrow \pi_{t,u} (kq^{tC_2}).$$
This spectral sequence is depicted in Figure ~\ref{ahkq} below. 

\begin{fig}\label{ahkq}
The Atiyah-Hirzebruch spectral sequence for $-9 \leq s \leq 1$ and $-9 \leq t \leq 9$ with some of the differentials drawn in.

The differentials are periodic and can be propogated downwards by examining the $A$-module structure of $H^{**}(\Sigma^{1,0} \underline{L}^2_{-9})$. A $\square$ represents $\z_{(2)}[\tau]$, a $\bullet$ represents $\f_2[\tau]$, and a red $\bullet$ represents $\f_2$. Differentials are blue and $\tau$-linear; a dashed differential means the source, target, or both are $\tau$-torsion. Green lines indicate hidden extensions. 

\hskip.5in \includegraphics[scale=.7]{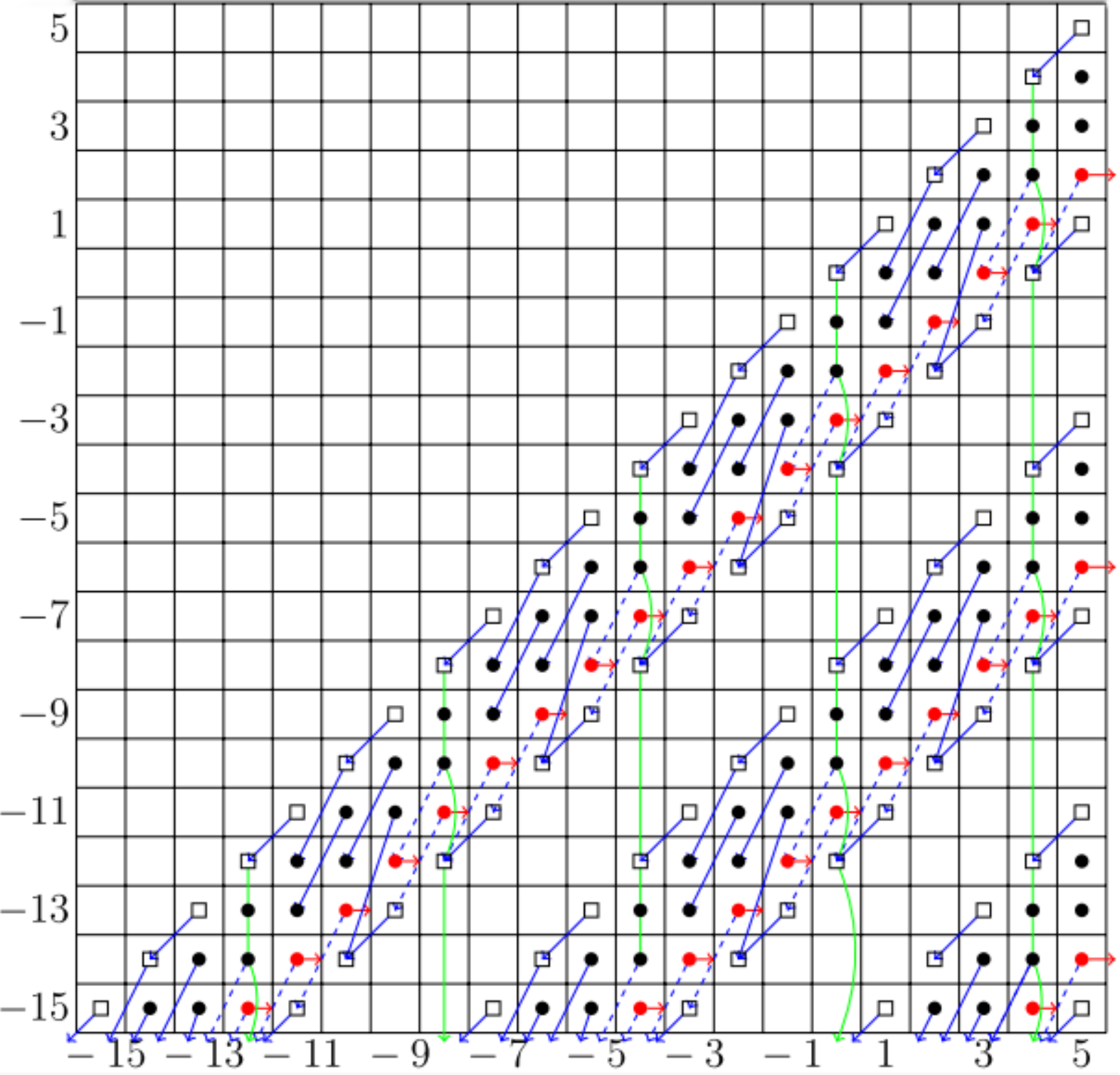}

\end{fig}

Note that $\eta$ is not nilpotent in $\pi_{**}(kq)$, so in each $s$-degree, an additional $\eta$-tower appears every $8$ $(t+s)$-degrees. The first two powers of $\eta$ always support a $\tau$-tower and the remaining elements of the tower are copies of $\z/2$. Further, the motivic weight of each copy of $kq_{**}$ changes by $1$ every $2$ $s$-degrees. 

Differentials in this spectral sequence are induced by the attaching maps in $\underline{L}^\infty_{-\infty}$ detected by the generators of $A(1)$. There are $d^1$-differentials between all the copies of $\z_2[\tau]$ so the targets represent $\f_2[\tau]$ on the $E^2$-page. There are $3$ $d^2$-differentials which correspond to the classical case; unlike the classical case, there is an additional differential from the element corresponding to $\eta^2$ to the element corresponding to $\eta^3$ two $s$-degrees lower. This differential propogates through the entire $\eta$-tower, but since $\eta^3 \tau = 0$, it only annihilates one copy of $\z/2$ in the copy of $\f_2[\tau]$ contributed by $\eta^2$. With this copy of $\f_2$ eliminated, the $\tau$-towers on $\eta^2$ in $s=1$ and the $\tau$-tower on $\alpha$ in $s=-2$ begin in the same motivic weight. Therefore the $d^3$-differentials between the towers completely annihilate them, and only copies of $\f_2[\tau]$ in the abutment.

Comparison with the Adams spectral sequence allows us to determine the hidden multiplication by $2$ extensions. The $E_2$-page of the inverse limit motivic Adams spectral sequence ~\cite[Section 4.2]{Gre12} computing $\pi_{**}(kq \wedge \Sigma^{1,0} \underline{L}^\infty_{-\infty})$ has the form
$$E^{***}_2 = \lim_{\underset{n}{\longleftarrow}} Ext^{***}_A(A//A(1) \otimes H^{**}(\Sigma^{1,0}\underline{L}^\infty_{-n}), \m_2).$$
By change-of-rings, we can rewrite this $E_2$-term as
$$\lim_{\underset{n}{\longleftarrow}} Ext^{***}_{A(1)}(H^{**}(\Sigma^{1,0}\underline{L}^\infty_{-n}),\m_2).$$
Therefore we just need to understand the action of $A(1)$ on $H^{**}(\Sigma^{1,0}\underline{L}^\infty_{-n})$. Recall that $Sq^{2^i}$ acts nontrivially on an element $x \in H^{n,\lceil n/2 \rceil}(\underline{L}^\infty_{-k})$ if and only if the $(i-1)$-st digit in the diadic expansion for $n$ is $1$. Using this fact we observe the following:
\begin{enumerate}
\item $Sq^1$ acts nontrivially on all cells with odd topological dimensions since they have diadic expansion ending in $1$
\item $Sq^2$ acts nontrivially on all cells with topological dimension congruent to $2,3 \mod 4$. 
\end{enumerate}
We can define a filtration of $Ext^{***}_{A(1)}(H^{**}_c(\Sigma^{1,0} \underline{L}^\infty_{-\infty}),\m_2)$ with associated graded consisting of suspensions of $A(1)//A(0)$ as follows. This filtration arises from a filtration of $H^{**}_c(\underline{L}^\infty_{-\infty})$ defined by setting $F_n \subset H^{**}_c(\underline{L}^{\infty}_{-\infty})$ to be the complement of $H^{2n,*}_c(\Sigma^{1,0} \underline{L}^\infty_{-\infty})$ inside of $H^{\leq 2n+1,*}_c(\underline{L}^\infty_{-\infty})$, i.e.
$$F_n := H^{\leq 2n+1,*}_c(\Sigma^{1,0} \underline{L}^\infty_{-\infty}) \setminus H^{2n,*}_c(\Sigma^{1,0} \underline{L}^\infty_{-\infty}).$$ 
Then each bidegree $(i,j)$ of $H^{**}_c(\underline{L}^\infty_{-\infty})$ where both generators of $A(1)$ act trivially contributes a copy of $Ext_{A(0)}(\m_2,\m_2) \cong \pi_{**}H\z$ in the resulting spectral sequence:
$$\bigoplus_{i \in \z} \Sigma^{4i,2i} Ext^{***}_{A(0)}(\Sigma^{1,0} \m_2,\m_2) \Rightarrow Ext^{***}_{A(1)}(H^{**}(\Sigma^{1,0} \underline{L}^\infty_{-\infty}),\m_2).$$
The differentials in this spectral sequence change tridegree by $(-1,-1,0)$, so there is no room for differentials. The resulting inverse limit motivic Adams spectral sequence collapses for tridegree reasons, so we obtain the desired extensions. 
\end{proof}

The same proof with $E(Q_0,Q_1)$ in place of $A(1)$ and $\pi_{**}(BPGL\langle 1 \rangle) \cong \z_2[\tau,v_1]$ applies \emph{mutatis mutandis} to prove the following corollary.

\begin{cor}
There is an isomorphism in homotopy groups
$$\pi_{**}(BPGL\langle 1 \rangle^{tC_2}) \cong  \lim_{\underset{n}{\longleftarrow}} \bigoplus_{i \geq -n} \pi_{**}(\Sigma^{4i,2i}BPGL\langle 0 \rangle).$$
\end{cor}

\begin{rem2}~\label{bptconj} It should be possible to apply the same proof using the Adams spectral sequence with $E(Q_0,\ldots,Q_n)$ in place of $A(1)$ and $\pi_{**}(BPGL\langle n \rangle) \cong \m_2[v_0,v_1,\ldots,v_n]$ in place of $\pi_{**}(kq)$ to produce an analogous splitting in homotopy groups for $BPGL\langle n  \rangle^{tC_2}$. In order to solve extensions, one can use induction on $n$ coupled with the induced maps of spectral sequences from the map $BPGL \langle n-1\rangle \to BPGL\langle n \rangle$. We conjecture that for any $n \geq 1$, there is an isomorphism in homotopy groups
$$\pi_{**}(BPGL\langle n \rangle^{tC_2}) \simeq \lim_{\underset{n}{\longleftarrow}} \bigoplus_{i \geq -n} \pi_{**}(\Sigma^{4i,2i}BPGL\langle n-1 \rangle).$$
This is a weak motivic analog of a conjecture of Davis-Mahowald ~\cite[Conjecture 1.6]{DM84}. 
\end{rem2}

\subsection{Splitting needed for computing $M(\eta^i)$} The material in this section will be needed to compute $M(\eta^i)$. We begin by recalling essential information about the category of modules over the cofiber of $\tau$ studied extensively in  ~\cite{Ghe17b}\cite{Ghe17}\cite{GWXPP}. The cofiber of the map $\tau \in \pi_{0,-1}(S^{0,0})$ is a $E_\infty$ motivic ring spectrum ~\cite{Ghe17} and therefore so is the $C\tau$-induced Eilenberg-Maclane spectrum $\overline{H} := H\f_2 \wedge C\tau$. The following proposition describes its cooperations and operations in the motivic stable homotopy category.

\begin{prop}~\cite[Propositions 5.4-5.5]{Ghe17}
The ring of cooperations of $\overline{H}$ is 
$$\pi_{**}(\overline{H} \wedge \overline{H}) \cong \f_2[\xi_1,\xi_2,\ldots,] \otimes E(\tau_0,\tau_1,\ldots) \otimes E(\beta_\tau)$$
where $\beta_\tau$ is the $\tau$-Bockstein with bidegree $(1,-1)$. The $\overline{H}$-Steenrod algebra is given by
$$\overline{H}^{**}(\overline{H}) \cong A/\tau \otimes E(\beta_\tau).$$
\end{prop}

This proposition simplifies when working in the category of $C\tau$-modules. The \emph{$C\tau$-linear $\overline{H}$-homology} of a $C\tau$-module is defined by taking $\overline{H}$-homology in the category of $C\tau$-modules, i.e.
$$\overline{H}_{**}(X) = \pi_{**}(\overline{H} \wedge_{C\tau} X).$$
By expanding the right-hand side, we see that
$$\overline{H}_{**}(X) \cong \pi_{**}(H \wedge C\tau \wedge_{C\tau} X) \cong H_{**}(X).$$
It turns out that working with $C\tau$-linear setting simplifies certain computations. For example, the $C\tau$-linear cooperations of $\overline{H}$ do not contain a $\tau$-Bockstein.

\begin{prop}~\cite[Proposition 2.5]{Ghe17b} \label{ohba}
The $C\tau$-linear cooperations of $\overline{H}$ are given by the Hopf algebra
$$\overline{A}_* \cong \f_2[\xi_1,\xi_2,\ldots] \otimes E(\tau_0,\tau_1,\ldots)$$
with bidegrees given by $|\xi_n| = (2^{n+1}-2,2^n-1)$ and $|\tau_n| = (2^{n+1}-1,2^n-1)$, and copoduct
$$\Delta(\xi_n) = \sum^n_{i=0} \xi^{2^i}_{n-i} \otimes \xi_i,$$
$$\Delta(\tau_n) = \tau_n \otimes 1 + \sum^n_{i=0} \xi^{2^i}_{n-i} \otimes \tau_i.$$
\end{prop}

\begin{rem2}
From now on, any time we use $\overline{H}$ or $\overline{A}$, we will be working in the $C\tau$-linear setting, i.e. the category of $C\tau$-modules. Note that any $C\tau$-module can be regarded as an $S^{0,0}$-module by composing with the inclusion of the bottom cell $S^{0,0} \to C\tau$.
\end{rem2}

Following ~\cite[Notation 2.11]{Ghe17b}, let $P_i$ be the dual of $\xi_{i}$, $i \geq 1$. Then $P_i$ is exterior and primitive in $\overline{A}$. In ~\cite{Ghe17b}, Gheorghe constructs an $E_\infty$ motivic ring spectrum $wBP$ satisfying
$$\overline{H}^{**}(wBP) \cong \overline{A}//E(P_1,P_2,\ldots)$$
which has homotopy groups given by
$$\pi_{**}(wBP) \cong \f_2[w_0,w_1,\ldots]$$
with $|w_i| = (2^{i+2}-3,2^{i+1}-1)$. 

The inclusion of the bottom cell $S^{0,0} \to C\tau$ induces a map $A \to \overline{A}$. Identifying classes in $\overline{A}$ with the images of the motivic Steenrod operations under this map, the $P_i$ can be defined inductively as the images of certain commutators ~\cite[Example 2.12]{Ghe17b}:
$$P_i = [Sq^{2^{i}},P_{i-1}].$$ 
In particular, $P_1 = Sq^2$ detects $\eta$ attaching maps in motivic cohomology and detects $w_0$ in the homotopy of $wBP$. 

In ~\cite{Ghe17b}, Gheorghe also constructs $E_\infty$ motivic ring spectra $wBP\langle n \rangle$ such that $\pi_{**}(wBP\langle n \rangle) \cong \f_2[w_0,\ldots,w_n]$. In order to compute $M(\eta^{4i})$, it suffices to produce a splitting of $\pi_{**}(wBP\langle 1 \rangle^{tC_2})$ as a wedge of suspensions of $\pi_{**}(wBP\langle 0 \rangle)$. However, we would like to compute $M(\eta^i)$ for all $i \geq 1$, so we must produce a new motivic spectrum. More precisely, we need a $C\tau$-module analog of classical $ko$ or motivic $kq$. To build this $C\tau$-module, we use the following result.

\begin{thm} ~\cite{GWXPP}
There is an equivalence of stable $\infty$-categories with $t$-structures
$$C\tau-mod^b_{cell} \to \cd^b(BP_*BP-comod)$$
whose restriction to the hearts is taking $BPGL$-homology. Here, $C\tau-mod^b_{cell}$ is the category of cellular module spectra over $C\tau$ whose $BPGL$-homology has bounded Chow degree, and $D^b(BP_*BP-comod)$ is the bounded derived category of the Abelian category of $p$-completed $BP_*BP$-comodules which are concentrated in even degrees.
\end{thm}

Using this theorem, we can define a $C\tau$-module which detects the classes we are interested in.

\begin{defin}
Recall that $BP_* \cong \z_{(2)}[v_1,\ldots]$ and $BP_*BP \cong BP_*[t_1,t_2,\ldots]$. We define $wko$ to be the $C\tau$-module corresponding to the $BP_*BP$-comodule 
$$\f_2[t^4_1,t^2_2,t_3,\ldots]$$
under the above equivalence of categories.
\end{defin}

\begin{rem2}
Analogously, one can define an exotic motivic analog of the motivic modular forms spectrum $mmf$ ~\cite{Ric17} by defining $wtmf$ to be the $C\tau$-module corresponding to the $BP_*BP$-comodule 
$$\f_2[t^8_1,t^4_2,t^2_3, t_4, \ldots]$$
under the above equivalence of categories. 
\end{rem2}

By construction, we have $\overline{H}^{**}(wko) \cong \overline{A}//\overline{A}(1)$ where $\overline{A}(1)$ is the subalgebra of $\overline{A}$ generated by $Sq^2, Sq^4$. To calculate its motivic homotopy groups, we will use the $C\tau$-linear $\overline{H}$-based Adams spectral sequence constructed in ~\cite[Section 2.4]{Ghe17b}. This spectral sequence has the form
$$Ext^{s,t,w}_{\overline{A}}(\overline{H}^{**}(X),\f_2) \Rightarrow [\Sigma^{t-s,w}C\tau,X]_{C\tau}.$$
Note that $\pi_{**}(X) \cong [\Sigma^{**}C\tau,X]_{C\tau}$ by the usual adjunction, so this spectral sequence computes the motivic homotopy groups of $X$.

\begin{lem}
The homotopy of $wko$ is 
$$\pi_{**}(wko) \cong \f_2[\eta,\nu, \alpha,\beta]/(\eta\nu,\nu^3,\nu\alpha,\alpha^2-\eta^2\beta)$$
where $|\eta| = (1,1)$, $|\nu| = (3,2)$, $|\alpha| = (11,7)$ and $|\beta| = (20,12)$. 
\end{lem}
\begin{proof}
The $C\tau$-linear $\overline{H}$-based Adams spectral sequence has the form
$$E_2 = Ext^{***}_{\overline{A}}(\overline{H}^{**}(wko),\f_2) \Rightarrow \pi_{**}(wko).$$
By change-of-rings, the $E_2$-page is isomorphic to
$$E_2 \cong Ext^{***}_{\overline{A}(1)}(\f_2,\f_2).$$
The subalgebra $\overline{A}(1)$ is identical to the subalgebra $A(1)$ of the classical Steenrod algebra if we replace $Sq^1$ by $Sq^2$ and replace $Sq^2$ by $Sq^4$. Therefore the $E_2$-page is obtained from the classical $Ext_{A(1)}(\f_2,\f_2)$ (which can be computed by using an $A(1)$-resolution) by the following algorithm:
\begin{enumerate}
\item For each classical $\f_2$ in bidegree $(s,t-s)$, place a copy of $\f_2$ in tridegree $(s,2t-s,t)$.
\item Replace $h_0$-extensions between classes in the classical $Ext$ with $h_1$-extensions between the corresponding classes in the new $Ext$.
\item Replace $h_1$-extensions between classes in the classical $Ext$ with $h_2$-extensions between the corresponding classes in the new $Ext$.
\end{enumerate}
Applying this algorithm, we arrive at the following $E_2$-page:
\begin{fig}
The $E_2$-page of the $C\tau$-linear $\overline{H}$-based Adams spectral sequence converging to $\pi_{**}(wko)$.

\hskip.2in \includegraphics[scale=.4]{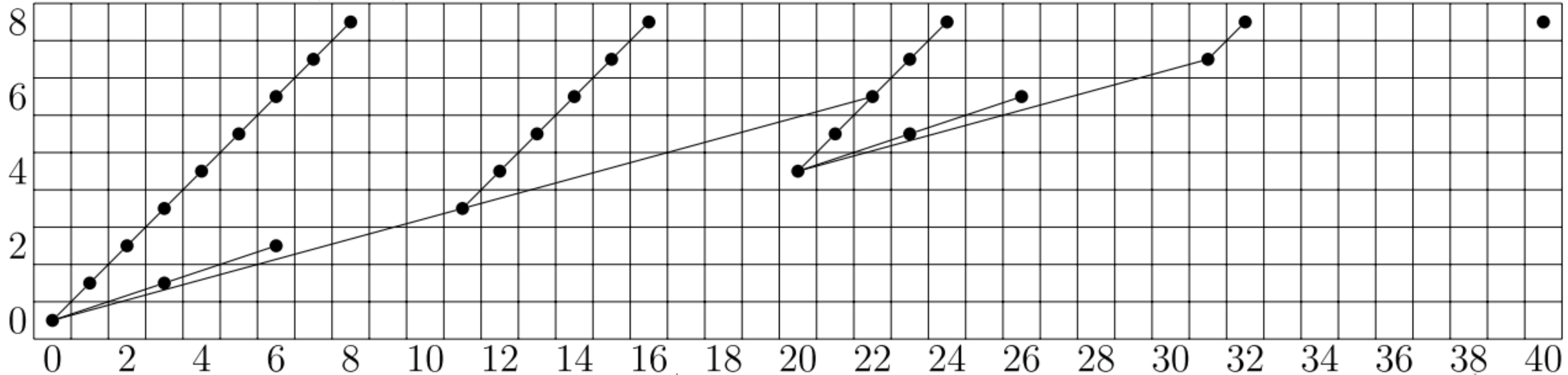}

A $\bullet$ represents $\f_2$. Lines of slope $1/2$ represent multiplication by $\eta$, lines of slope $1/3$ represent multiplication by $\nu$, and lines of slope $3/11$ represent multiplication by $\alpha$. 
\end{fig}
There is no room for differentials, so we obtain the desired isomorphism. 
\end{proof}

\begin{prop}
There is an isomorphism
$$\pi_{**}(wko^{tC_2}) \cong \lim_{\underset{n}{\longleftarrow}} \bigoplus_{i \geq -n} \left( \Sigma^{8i-1,4i-1} \pi_{**}(wBP\langle 0\rangle) \oplus \Sigma^{8i,4i} \pi_{**}(wBP\langle 0\rangle) \right).$$
\end{prop}

\begin{proof}
The Atiyah-Hirzebruch spectral sequence arising from the cellular filtration of $\Sigma^{1,0} \underline{L}^\infty_{-\infty}$ has the form
$$E^1_{s,t,u} = wko_{t,u}(S^{s,\lfloor s/2 \rfloor}) = \pi_{t+s,u+\lfloor s/2 \rfloor}(wko) \Rightarrow \pi_{s+t,u+\lceil s/2 \rceil} ( wko^{tC_2}).$$
This spectral sequence is depicted in the Figure ~\ref{ahwko} below.

The differentials can be read off from the action of $\overline{A}(1) = \langle Sq^2,Sq^4 \rangle$ on $H^{**}(\underline{L}^\infty_{-\infty})$. This action is given by
$$Sq^{2i}(v^k) = {2k\choose2i} v^{k+i}$$
$$Sq^{2i}(uv^k) = {2k\choose2i} uv^{k+i}.$$
In particular, we observe the following:
\begin{enumerate}
\item $Sq^2$ acts nontrivially on all cells in topological dimension congruent to $2,3 \mod 4$. 
\item $Sq^4$ acts nontrivially on all cells in topological dimension congruent to $4,5,6,7 \mod 8$.
\end{enumerate}

To solve extensions, we use the inverse limit $C\tau$-linear $\overline{H}$-based Adams spectral sequence, where the inverse limit is taken over the negative skeleta of $\Sigma^{1,0} \underline{L}^\infty_{-n}$. Then we have
$$E^{***}_2 = \lim_{\underset{n}{\longleftarrow}}  Ext^{***}_{\overline{A}}(\overline{A}//\overline{A}(1) \otimes \overline{H}^{**}(\Sigma^{1,0} \underline{L}^\infty_{-n}),\f_2) \Rightarrow \pi_{**}(wko^{tC_2}).$$
Here we have used that $C\tau$-linear $\overline{H}$-cohomology satisfies the strong K\"unneth theorem for smash products over $S^{0,0}$ ~\cite[Proposition 2.9]{Ghe17b}. By change-of-rings, we can rewrite the $E_2$-term as
$$\lim_{\underset{n}{\longleftarrow}} Ext_{\overline{A}(1)}(\overline{H}^{**}(\Sigma^{1,0} \underline{L}^\infty_{-n}),\f_2)$$
so we must understand the action of $Sq^2$ and $Sq^4$ on $\overline{H}^{**}(\Sigma^{1,0} \underline{L}^\infty_{-n})$. This cohomology is $\overline{H}^{**}(\Sigma^{1,0} \underline{L}^\infty_{-n}) = \pi_{**}(H\f_2 \wedge C\tau \wedge \Sigma^{1,0} \underline{L}^\infty_{-n}) \cong \Sigma^{1,0} \f_2[u,v,v^{-1}]/(u^2)$ since smashing with $C\tau$ kills multiplication by $\tau$ ~\cite[Lemma 5.3]{Ghe17}. The map $\Sigma^{1,0} \underline{L}^\infty_{-n}\to C\tau \wedge \Sigma^{1,0} \underline{L}^\infty_{-n} $ sends $\Sigma^{1,0} u$ and $\Sigma^{1,0} v$ to the generators of the same name, so we can compute the action of the generators of $\overline{A}(1)$ on $\overline{H}^{**}(\Sigma^{1,0} \underline{L}^\infty_{-n})$ using the relations above. Therefore the action of $\overline{A}(1)$ is nontrivial on the cells in topological dimension congruent to $6,7 \mod 8$. Each such cell corresponds to a direct summand in the following algebraic spectral sequence with associated graded consisting of suspensions of $\overline{A}(1)//\overline{A}(0)$, which can be constructed like the analogous spectral sequence in the proof of Proposition ~\ref{buh}:
$$\bigoplus_{i \in \z} \Sigma^{8i-1,4i-1} Ext_{\overline{A}(0)}(\f_2) \oplus \Sigma^{8i,4i} Ext_{\overline{A}(0)}(\f_2) \Rightarrow Ext_{\overline{A}(1)}(\overline{H}^{**}_c(\Sigma^{1,0} \underline{L}^\infty_{-\infty})).$$
The differentials in this spectral sequence change (topological, homological) bidegree by $(-1,-1)$ and preserve motivic weight, so the spectral sequence collapses. The direct sum of resulting Adams spectral sequences also collapses, producing Figure ~\ref{ASSwko} below which solves the extensions in the Atiyah-Hirzebruch spectral sequence.

\begin{fig}\label{ahwko}
The following is the Atiyah-Hirzebruch spectral sequence for $-20 \leq s \leq 1$ and $-20 \leq t \leq 20$ with some of the differentials drawn in. These differentials are periodic and can be propogated downwards by examining the $\overline{A}(1)$-module structure of $H^{**}(\underline{L}^{5}_{-10})$. A $\square$ represents $\f_2[\eta]$, and a $\bullet$ represents $\f_2$. 

\includegraphics[scale=.7]{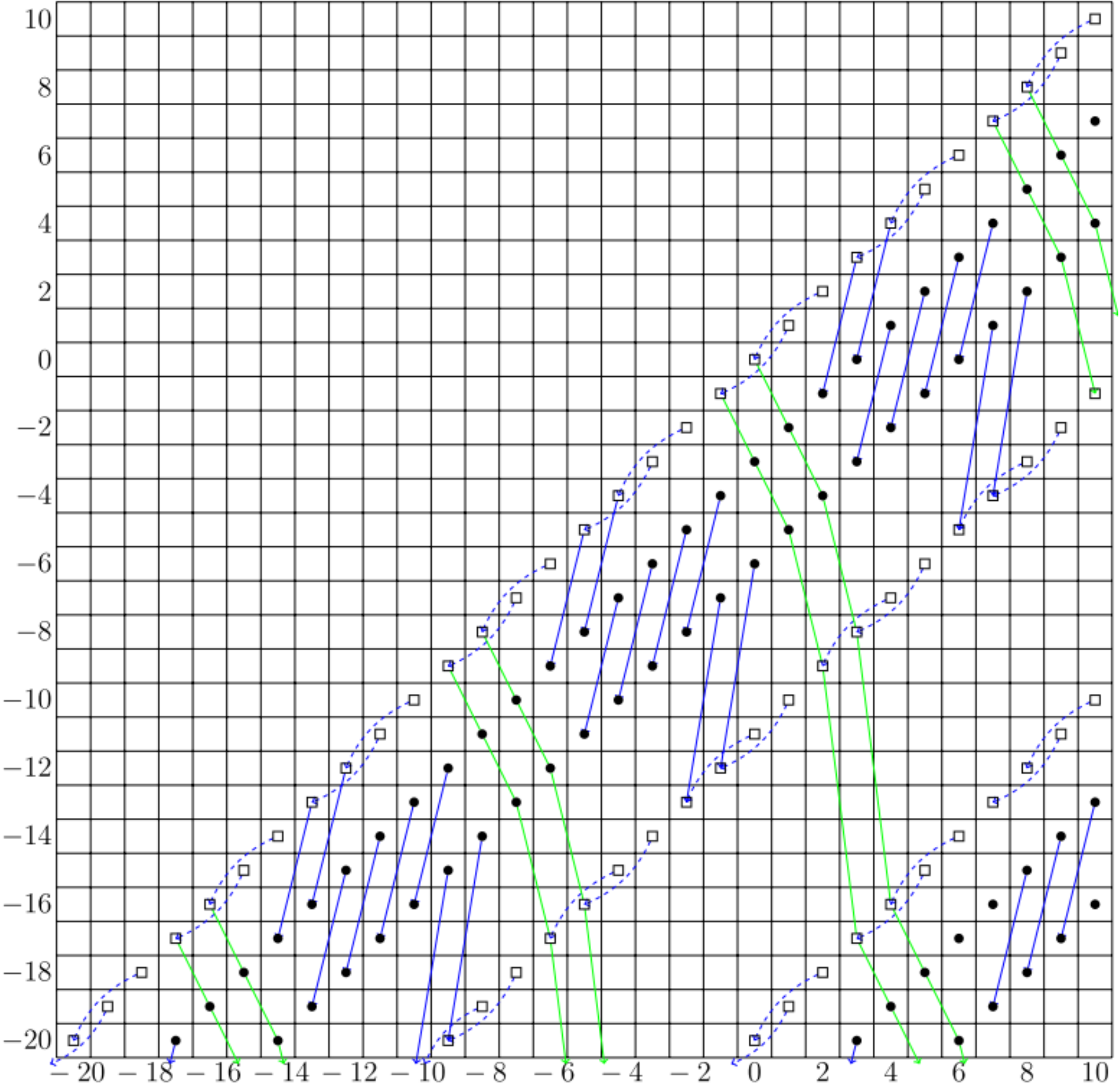}

\end{fig}

\begin{fig}\label{ASSwko}
The following is the inverse limit $C\tau$-linear $\overline{H}$-based Adams spectral sequence for $wko^{tC_2}$ for $-9 \leq t-s \leq 8$ and $0 \leq s \leq 5$. 

\hskip1.2in \includegraphics[scale=.3]{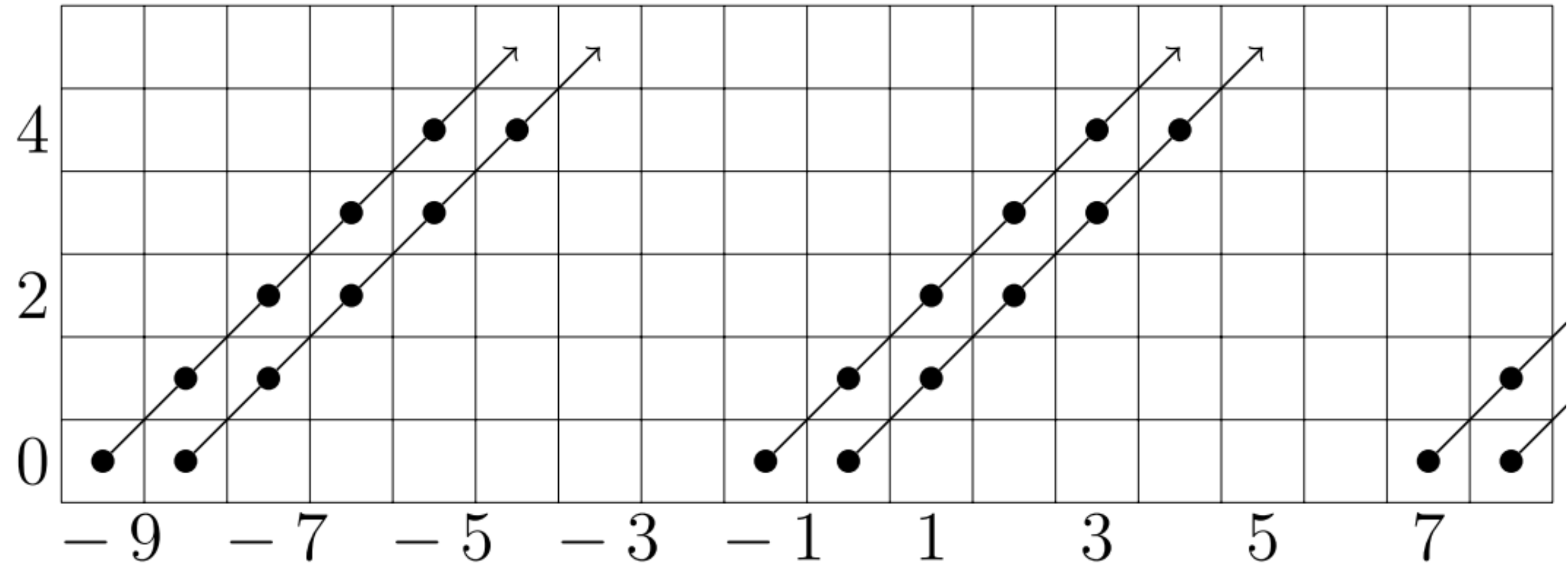}

\end{fig}

\end{proof}

As in the case of $kq$ and $BPGL\langle 1 \rangle$, the same proof applies \emph{mutatis mutandis} to produce a splitting in homotopy groups for $wBP\langle 1 \rangle^{tC_2}$. 

\begin{cor}
There is an isomorphism in homotopy groups
$$\pi_{**}(wBP\langle 1 \rangle^{tC_2}) \cong \bigoplus_{i \in \z} \left( \Sigma^{4i-1,2i-1} \pi_{**}(wBP\langle 0 \rangle) \oplus \Sigma^{4i,2i} \pi_{**}(wBP\langle 0\rangle) \right)$$
\end{cor}

\begin{rem2}
An analogous computation where we replace $E(P_0,P_1)$ by $E(P_0,P_1,\ldots,P_n)$, along with an induction on $n$ to solve extensions should produce a similar splitting in homotopy groups for $wBP\langle n \rangle^{tC_2}$. We conjecture in analogy with Remark ~\ref{bptconj} that for any $n \geq 1$, there is an isomorphism in homotopy groups
$$\pi_{**}(wBP\langle n \rangle^{tC_2}) \cong \bigoplus_{i \in \z} \left( \Sigma^{4i,2i} \pi_{**}(wBP\langle n-1\rangle \oplus \Sigma^{4i-1,2i-1} \pi_{**}(wBP\langle n-1 \rangle) \right).$$
\end{rem2}

\section{Computations of the $E$-Mahowald invariant} 

The classical Mahowald invariant is conjectured to produce redshift, i.e. the Mahowald invariant of a $v_n$-periodic element $\alpha \in \pi_*(S^0)$ should be $v_n$-torsion. In this section, we present some evidence that the motivic Mahowald invariant of a $v_n$-periodic element is $v_{n+1}$-periodic and the motivic Mahowald invariant of a $w_n$-periodic element is $w_{n+1}$-periodic. 

The following is a motivic analog of ~\cite[Theorem 2.16]{MR93}. We will use this in the following section to compute $M(2^i)$. 

\begin{prop}\label{kq2}
For any $a,b \geq 0$, we have
$$M_{kq}(2^{4a+b}) \ni \begin{cases}
\beta^a \quad &b =0, \\
\eta \beta^a & b=1,\\
\eta^2 \beta^a & b=2, \\
\alpha \beta^a & b=3 .
\end{cases}$$
\end{prop}
\begin{proof} 
We must determine the minimal $N > 0 $ such that the left-hand composition is nontrivial in the following diagram:
\[
\begin{tikzcd}
S^{0,0} \arrow[r,dashed] \arrow{d}{2^{4a+b}} & \Sigma^{-2N+1,-N} kq \vee \Sigma^{-2N+2,-N+1}kq \arrow{d}\\
kq^{tC_2} \arrow{r}& \Sigma^{1,0}kq \wedge \underline{L}^\infty_{-N}
\end{tikzcd}
\]
One can show from the definition that $N = \lceil -s+1/2\rceil$ where $s$ is the cellular filtration degree of the element detecting $2^{4a+b} \in \pi_{**}(H\z)$ in the Atiyah-Hirzebruch spectral sequence computing $\pi_{**}(kq^{tC_2})$. For small values of $a,b$, we can read off the desired degrees $s$ from the picture of this Atiyah-Hirzebruch spectral sequence in the previous section. We obtain the following table of values:
\begin{center}
\begin{tabular}{ c c c c }
$x$ & s & N & $M_{kq}(x)$\\
 2 & -1 & 1 & $\eta$\\ 
4 & -2 & 2 & $\eta^2$\\  
 8 & -4 & 3 & $\alpha$\\
16 & -8 & 5   & $\beta$
\end{tabular}
\end{center}
This spectral sequence is $8$-periodic in the $s$-direction, so the desired result follows from this low-degree computation.
\end{proof}

The same proof using the Atiyah-Hirzebruch spectral sequence for $BPGL\langle 1 \rangle$ proves the following:

\begin{cor}
For any $i \geq 1$, we have $v^i_1 \in M_{BPGL\langle 1 \rangle}(v^i_0)$. 
\end{cor}

\begin{rem2}~\label{vnrmk} Assuming the conjectured splittings in homotopy groups from the previous section, the previous proof should generalize to prove that redshift for $v_{n-1}$ occurs in the approximation to the motivic Mahowald invariant using $E=BPGL\langle n \rangle$. In particular, we expect that for any $n \geq 1$, we have $v^i_n \in M_{BPGL\langle n \rangle}(v^i_{n-1})$.
\end{rem2}

The following result will be used to compute $M(\eta^i)$. 

\begin{prop}\label{wkoeta}
For any $a,b \geq 0$, we have
$$M_{wko}(\eta^{4a+b}) \ni \begin{cases}
\beta^a \quad &b =0 ,\\
\nu \beta^a & b=1,\\
\nu^2 \beta^a & b=2,\\
\alpha \beta^a & b=3 .
\end{cases}$$
\end{prop}
\begin{proof}
We must determine the minimal $N > 0 $ such that the left-hand composition is nontrivial in the following diagram:
\[
\begin{tikzcd}
S^{i,i} \arrow[r,dashed] \arrow{d}{\eta^{4a+b}} & \Sigma^{-2N+1,-N} wko \vee \Sigma^{-2N+2,-N+1}wko \arrow{d}\\
wko^{tC_2} \arrow{r}& \Sigma^{1,0}wko \wedge \underline{L}^\infty_{-N}
\end{tikzcd}
\]
One can show from the definition that $N = \lceil -s+1/2\rceil$ where $s$ is the cellular filtration degree of the element detecting $\eta^{4a+b} \in \pi_{**}(wBP\langle 0 \rangle)$ in the Atiyah-Hirzebruch spectral sequence computing $\pi_{**}(wko^{tC_2})$. For small values of $a,b$, we can read off the desired degrees $s$ from the picture of this Atiyah-Hirzebruch spectral sequence in the previous section. We obtain the following table of values:
\begin{center}
\begin{tabular}{ c c c c }
$x$ & s & N & $M_{wko}(x)$\\
 $\eta$ & -2 & 2& $\nu$\\ 
$\eta^2$ & -4 & 3 & $\nu^2$\\  
 $\eta^3$ & -8 & 5 & $\alpha$\\
$\eta^4$ & -16 & 9  & $\beta$
\end{tabular}
\end{center}
This spectral sequence is $(16,4)$-periodic in the $(s,t+s)$-direction, so the desired result follows from this low-degree computation.
\end{proof}

The same proof using the Atiyah-Hirzebruch spectral sequence for $wBP\langle 1 \rangle^{tC_2}$ proves the following.

\begin{cor}
For any $i \geq 1$, we have $w^i_1 \in M_{wBP\langle 1 \rangle}(w^i_0)$. 
\end{cor}

\begin{rem2}
As in the Remark ~\ref{vnrmk}, this technique should generalize in view of the conjectured splitting of the motivic $C_2$-Tate construction of $wBP\langle n \rangle$ to prove that for any $n \geq 1$, we have $w^i_n \in M_{wBP\langle n \rangle}(w^i_{n-1})$. 
\end{rem2}

\section{Motivic Mahowald invariants of $2^i$ and $\eta^i$} 

In this section we compute the motivic Mahowald invariants of $2^i$ and $\eta^i$ for all $i \geq 1$. We recover the classical computation of Mahowald-Ravenel by showing that the first element of Adams' $v_1$-periodic family in Adams filtration $i$ is contained in $M(2^i)$, and we show that the first element of Andrews' $w_1$-periodic family in Adams filtration $i$ is contained in $M(\eta^i)$. 

\subsection{Atiyah-Hirzebruch-May names}
If $X$ is a finite cell complex, we will often refer to the Atiyah-Hirzebruch-May names of elements $\alpha \in \pi_{**}(X)$ in the sequel. In this subsection, we establish terminology and notation for this convention. There are three spectral sequences involved in assigning an Atiyah-Hirzebruch-May name, and these depend on whether or not we are in the $C\tau$-linear setting:
\begin{enumerate}
\item In the usual motivic setting, we use the spectral sequences
$$\text{motivic May SS} \rightsquigarrow \text{motivic Adams SS} \rightsquigarrow \text{Atiyah-Hirzebruch SS}.$$
\item In the $C\tau$-linear setting, we use the spectral sequences
$$\text{$C\tau$-linear $\overline{H}$-based May SS} \rightsquigarrow \text{C$\tau$-linear $\overline{H}$-based Adams SS} \rightsquigarrow \text{Atiyah-Hirzebruch SS}.$$
\end{enumerate}

First, suppose we are in the usual motivic setting. Let $\alpha \in \pi_{**}(X)$. Then $\alpha$ is detected by some class $x[(m,n)]$ in the $E^\infty$-page of the Atiyah-Hirzebruch spectral sequence arising from the cellular filtration of $X$. The filtration quotients have the form $\pi_{**}(S^{i,j})$; the notation above indicates that $x$ comes from the filtration quotient $\pi_{**}(S^{m,n})$. Now, $x \in \pi_{**}(S^{m,n}) \cong \pi_{*-m,*-n}(S^{0,0})$ is detected by some class $y$ in the $E_\infty$-page of the motivic Adams spectral sequence. Therefore we can say that $\alpha$ is detected by $y[(i,j)]$. Finally, we can regard $y$ as a class in the $E_2$-page of the motivic Adams spectral sequence $E_2 = Ext^{***}_{A}(\m_2,\m_2)$, so $y$ is detected by some class $z$ in the $E^\infty$-page of the motivic May spectral sequence ~\cite[Section 5]{DI10}. We say that the \emph{Atiyah-Hirzebruch-May name} of $\alpha$ is $z[(m,n)]$. The $E_1$-page of the motivic May spectral sequence is generated over $\f_2$ by classes $\{\tau, h_{ij} : i > 0, j \geq 0\}$. Therefore a typical Atiyah-Hirzebruch-May name will have the form $z[(m,n)]$ where $z$ is a polynomial in the $h_{ij}$'s. Note that if we replace $\pi_{**}(X)$ by $kq_{**}(X)$, this naming procedure is still valid since $A(1)_*$ is a quotient of the dual motivic Steenrod algebra so the motivic May spectral sequence makes sense.

Before defining the Atiyah-Hirzebruch-May name in the $C\tau$-linear setting, we need a $C\tau$-linear $\overline{H}$-based May spectral sequence. 

\begin{lem}
There is a $C\tau$-linear $\overline{H}$-based May spectral sequence with $E_1$-term given by
$$E_1 = \f_2[h_{i,j} : i>0,j\geq 0] \Rightarrow Ext^{s,(b,c)}_{\overline{A}} (\f_2,\f_2).$$
\end{lem}

\begin{proof}
We define an increasing filtration of $\overline{A}_*$ as in ~\cite[Theorem 3.2.3]{Rav86}. Set $|\tau_i| = 2(i+1)-1$ and $|\xi^{2^j}_i| = 2i-1$ and define $h_{i,0} = \tau_{i-1}$ for $i \geq 1$ and $h_{i,j} = \xi^{2^{j-1}}_i$ for $i \geq 1$ and $j \geq 1$. Then each $h_{i,j}$ is primitive in the associated graded algebra under the coproduct described in Proposition ~\ref{ohba}, so we obtain a May spectral sequence with 
$$E_1 = \f_2[h_{ij} : i >0, j \geq 0]$$
by ~\cite[Lemma 3.1.9]{Rav86}. The $d_1$-differentials are given by
\[
d_1(h_{i,j}) = \begin{cases}
h_{i,0} + \sum^{i-1}_{k=0} h_{i-1-k,k+1}h_{i+1,0} \quad &\text{ if } j=0,\\
\sum^i_{k=0} h_{n-i,i+j+1} h_{i,j+1} \quad &\text{ else.}
\end{cases}
\]
\end{proof}

\begin{rem2}
The discussion of the motivic May spectral sequence from ~\cite[Section 5]{DI10} carries over \emph{mutatis mutandis}. We will freely use their results when they can be translated to the $C\tau$-linear setting.
\end{rem2}

Now, suppose we are in the $C\tau$-linear setting. Let $\alpha \in \pi_{**}(X)$. By repeating the procedure from the usual motivic setting above with the $C\tau$-linear $\overline{H}$-based May and Adams spectral sequences, we obtain a well-defined Atiyah-Hirzebruch-May name for $\alpha$. Note that if we replace $\pi_{**}(X)$ by $wko_{**}(X)$, this naming procedure is still valid since $\overline{A}(1)_*$ is a quotient of $\overline{A}_*$.

\subsection{Periodicity operators}
In the sequel, we will use motivic analogs of the homological periodicity operator defined by Adams in ~\cite{Ada66b}. The following is a specialization of ~\cite[Corollary 5.5]{Ada66b} to the case $r=2$. 

\begin{lem}
In the classical Adams spectral sequence, the Massey product
$$P_v(-) := \langle h_3, h^4_0, - \rangle$$
induces an isomorphism
$$Ext^{s,t}_A(\f_2,\f_2) \overset{\cong}{\to} Ext^{s+4, t+12}_A(\f_2,\f_2)$$
when $1 < s < t < \min(4s-2,2+2^{3}+T(s-2))$ where $T(s)$ is the numerical function defined by
$$T(4k) = 12k, \quad T(4k+1) = 12k+2, \quad T(4k+2) = 12k+4, \quad T(4k+3) = 12k+7.$$
This Massey product detects multiplication by $v^4_1$ in the classical stable stems. 
\end{lem}

Some discussion of Massey products in the motivic Adams spectral sequence can be found in ~\cite[Section 4.4]{DI10}. We now state the motivic analog of the previous lemma, which can be proven using Levine's Theorem ~\cite{Lev14} to verify the nontriviality of the relevant Massey products. 

\begin{cor}\label{mv1per}
In the motivic Adams spectral sequence, the Massey product
$$P_v(-) := \langle h_3, h^4_0, - \rangle$$
induces a map
$$Ext^{s,t,u}_A(\m_2,\m_2) \to Ext^{s+4, t+12,u+4}_A(\m_2,\m_2)$$
which is an isomorphism when $1 < s < t < \min(4s-2,2+2^{3}+T(s-2))$ and $u \leq s$. 
\end{cor}

\begin{rem2}
Alternatively, one can replace $\m_2$ by $H^{**}(C\eta)$ to remove the restriction $u \leq s$. In both cases, the additional condition is required to avoid certain $h_1$-local classes described in ~\cite{GI15}. For example, the class $h^6_1 d_0$ in $Ext^{10,20,14}_A(\m_2,\m_2)$ lies in the $(s,t)$-range stated in the corollary, but is not in the image of $P_v(-)$. We thank Dan Isaksen for pointing out this issue and suggesting a solution. 
\end{rem2}

If $\alpha \in \pi_{**}(S^{0,0})$ is detected by an element $P^k_v(x)$ in the motivic Adams spectral sequence and $x$ detects $\beta \in \pi_{**}(S^{0,0})$, then we will denote $\alpha$ by $P^k_v(\beta)$. 

We also need a version of this periodicity operator for detecting $w^4_1$-periodicity. This arises from the classical periodicity operator via the following lemma.

\begin{lem}\label{hv1per}
The doubling homomorphism 
$$\cd: Ext^{**}_{A_{cl}}(\f_2,\f_2) \to Ext^{***}_{\overline{A}}(\f_2,\f_2)$$
induced by the map $A_{cl} \to \overline{A}$ defined by $Sq^k \mapsto Sq^{2k}$ is an isomorphism onto its image which preserves all higher structure, including products, squaring operations, and Massey products. Here $A_{cl}$ denotes the classical Steenrod algebra.
\end{lem}

\begin{proof}
We need some results from ~\cite[Section 2.1.3]{Isa14}. Let $A'$ be the subquotient $\m_2$-algebra of $A$ generated by $Sq^{2k}$ for all $k\geq0$, subject to the relation $\tau=0$. Then the doubling homomorphism defined as above is an isomorphism $A_{cl} \to A'$. By ~\cite[Theorem 2.1.12]{Isa14}, this induces an isomorphism between the classical Adams $E_2$-page $Ext_{A_{cl}}(\f_2,\f_2)$ and the subalgebra of the motivic Adams $E_2$-page $Ext_{A}(\m_2,\m_2)$ consisting of elements in degrees $(s,f,w)$ such that $s+f-2w=0$. Here, $s$ is the stem, $f$ is the Adams filtration, and $w$ is the motivic weight. This isomorphism preserves all higher structure. 

Since the relation $\tau=0$ is already imposed on $\overline{A}$, our lemma follows trivially from the previous remarks.
\end{proof}

\begin{rem2}~\label{bop}
By the proof of ~\cite[Theorem 2.1.12]{Isa14}, the above map can be described using May names by
$$h_{i,j-1} \mapsto h_{i,j}.$$
\end{rem2}

\begin{cor}\label{mw1per}
In the $C\tau$-linear $\overline{H}$-based Adams spectral sequence, the Massey product
$$g(-) := \langle h_4, h^4_1, - \rangle$$
induces a map
$$Ext^{s,t,u}_{\overline{A}}(\f_2,\f_2) \overset{\cong}{\to} Ext^{s+4, t+20,u+12}_{\overline{A}}(\f_2,\f_2).$$
If $P_v(x)$ is defined with no indeterminacy in the classical Adams spectral sequence, then $g(\cd(x))$ is defined with no indeterminacy in the $C\tau$-linear $\overline{H}$-based Adams spectral sequence.
\end{cor}

We note that the Massey product $g(-)$ detects multiplication by $g$ in the $C\tau$-linear $\overline{H}$-based Adams spsectral sequence. We will discuss the fate of the classes in the $C\tau$-linear $\overline{H}$-based Adams spectral sequence arising from iterating this Massey product when we compute the motivic Mahowald invariant of $\eta^i$.

\subsection{Motivic Mahowald invariant of $2^i$}
We now compute $M(2^i)$. We will state our result in terms of iterated Toda brackets; we verify the nontriviality of these Toda brackets in the following lemma. 

\begin{lem}
The iterated Massey products $P^i_v(x)$ for $x = 8\sigma, \eta, \eta^2, \eta^3$ are permanent cycles in the motivic Adams spectral sequence. Moreover, the Betti realizations of the classes which they detect in $\pi_{**}(S^{0,0})$  are the classes in $\pi_{*}(S^{0})$ detected by the iterated Massey products $P^i(x)$ for $x = 8\sigma, \eta, \eta^2, \eta^3$, where $P(-)$ is the classical Adams periodicity operator.
\end{lem}

\begin{proof}
The iterated Massey products $P^i_v(x)$ are non-vanishing since their classical analogs are non-vanishing by ~\cite{Ada66b}. Denote these classical analogs by $P^i(x)$. The classes $P^i(x)$ are permanent cycles in the classical Adams spectral sequence by ~\cite{Ada66}. Betti realization induces a map of spectral sequences from the motivic Adams spectral sequence to the classical Adams spectral sequence which sends $h_{ij} \mapsto h_{ij}$ and therefore carries $P^i_v(x)$ to $P^i(x)$. Therefore the images of $P^i_v(x)$ are permanent cycles, so the classes $P^i_v(x)$ are permanent cycles. 
\end{proof}

\begin{thm}\label{mm2i}
Let $i \geq 1$. The motivic Mahowald invariant of $2^i$ is given by
\[ M(2^i) \ni \begin{cases}
P^{\lfloor i/4 \rfloor}_v(h^3_0 h_3) \quad & i \equiv 0 \mod 4 ,\\
P^{\lfloor i/4 \rfloor}_v(h_1) & i \equiv 1 \mod 4,\\
P^{\lfloor i/4 \rfloor}_v(h_1^2) & i \equiv 2 \mod 4,\\
P^{\lfloor i/4 \rfloor}_v(h_1^3) & i \equiv 3 \mod 4.
\end{cases}
\]
Here we are denoting nontrivial Toda brackets in $\pi_{**}(S^{0,0})$ by the Massey products which detect them.
\end{thm}

\begin{proof}
Nontriviality of the relevant Massey products was proven in the previous lemma. We break the proof apart into two cases depending on the congruence of $i \mod 4$.

($i \equiv 1,2 \mod 4$) The elements $\beta^k\eta,\beta^k\eta^2 \in \pi_{**}(kq)$ are in the Hurewicz image for $kq$. Their inverse images are $P^k(\eta), P^k(\eta^2) \in \pi_{**}(S^{0,0})$, so this case is clear by Proposition ~\ref{lift} applied to Proposition ~\ref{kq2}. 

($i \equiv 0,3 \mod 4$) The elements $\alpha, \beta \in \pi_{**}(kq)$ are not in the Hurewicz image, so we cannot immediately employ Proposition ~\ref{lift}. Instead, we must pass through the motivic Mahowald invariant based on the mod $2$ Moore spectrum $V(0)$. We obtain the theorem through the following series of approximations:
$$M_{kq}(2^i) \rightsquigarrow M_{kq}(2^i;V(0)) \rightsquigarrow M(2^i;V(0)) \rightsquigarrow M(2^i).$$
\begin{enumerate}
\item By Proposition ~\ref{kq2}, we have computed $M_{kq}(2^i)$
\item By the proof of Proposition ~\ref{kq2}, the coset $M_{kq}(2^i)$ is detected on the cell of $\Sigma^{1,0} \underline{L}^\infty_{-\infty}$ with topological degree $f(i)$ where for each $k \geq 0$ we define
$$ f(4k) = -8k , \quad \quad f(4k+3) = -4-8k.$$
The topological degree of the cell determines where in the filtration of $\underline{L}^\infty_{-\infty}$ by $V(0)$ the coset $M_{kq}(2^i;V(0))$ is detected, i.e. it determines the $N$ in Definition ~\ref{miebx}. We obtain the following table of values:
\begin{center}
\begin{tabular}{ c c c c }
$x$ & s & N & $M_{kq}(x;V(0))$\\
$2^{4k}$ & $-8k$ & $1+4k$   & $\beta^{k}[(0,0)]$\\
 $2^{4k+3}$ & $-4-8k$ & $3+4k$ & $\beta^k \alpha[(0,0)]$.
\end{tabular}
\end{center}
To justify that $\beta^k \alpha^\epsilon[(0,0)]$, $\epsilon \in \{0,1\}$, are the correct Atiyah-Hirzebruch-May names for the classes above, it suffices to note that $\beta^k \alpha^\epsilon \in kq_{**}$ is not in the image of multiplication by $2$ for any $k, \epsilon$. 

\item Let $h : S^{0,0} \wedge V(0) \to kq \wedge V(0)$ be the Hurewicz map smashed with $id_{V(0)}$. We will prove the following by induction on $k$:
\begin{center}
\begin{tabular}{ c c  }
$x$ & $h^{-1}(x)$\\
$\beta^{k}[(0,0)]$ & $P^{k-1}_v(8\sigma)[(1,0)]$\\
$\beta^k \alpha[(0,0)]$ & $P^k_v(\eta^3)[(1,0)].$ 
\end{tabular}
\end{center}

We start with the case $x = \beta^k[(0,0)]$, $k=1$. Recall that $V(0)$ is defined by the cofiber sequence
$$S^{0,0} \overset{\cdot 2}{\to} S^{0,0} \to V(0),$$
and this cofiber sequence gives rise to a long exact sequence in homotopy
$$\cdots \to \pi_{**}(S^{0,0}) \overset{f_*}{\to} \pi_{**}(S^{0,0}) \overset{j_*}{\to} \pi_{**}(V(0)) \overset{\delta}{\to} \pi_{*-1,*}(S^{0,0}) \to \cdots$$
where $f_*$ is induced by $\cdot 2$. There is a motivic May differential  $d_4(b^2_{20}) = h^4_0h_3$ which produces the relation $2 \cdot (8\sigma) = 0$ in the motivic Adams spectral sequence converging to $\pi_{**}(S^{0,0})$ \cite[Section 5.3]{DI10}. Therefore $8\sigma \in ker(f_*) = im(\delta_*)$, so the Atiyah-Hirzebruch-May name of $\delta^{-1}(8\sigma) \in \pi_{**}(V(0))$ is $h^3_0h_3[(1,0)]$. If we smash the above cofiber sequence with $kq$ and take homotopy groups, the class $h^3_0 h_3$ is trivial since $h_3$ does not appear in the motivic May spectral sequence converging to $Ext_{A(1)}(\m_2,\m_2)$. Therefore the class $h^3_0 h_3[(1,0)]$ detects zero in $\pi_{**}(kq \wedge V(0))$.  Note that the above May differential implies that $b^2_{20}$ detects zero in $\pi_{**}(S^{0,0})$, so $b^2_{20}[(0,0)]$ detects zero in $\pi_{**}(V(0))$. 

On the other hand, the class $b^2_{20}$ detects $\beta \in \pi_{**}(kq)$ by the motivic analog of the classical argument that $b^2_{20}$ detects $\beta \in \pi_*(ko)$. Smashing the cofiber sequence defining $V(0)$ with $kq$ and applying homotopy produces another long exact sequence
$$\cdots \to kq_{**} \overset{f_*}{\to} kq_{**} \overset{j_*}{\to} kq_{**}(V(0)) \overset{\delta}{\to} kq_{*-1,*} \to \cdots.$$
Since $\alpha \neq 2y$ for any $y \in \pi_{**}(kq)$, we have that $\alpha \notin im(f_*) = ker(j_*)$. Therefore the Atiyah-Hirzebruch-May name of $j_*(\alpha) \in \pi_{**}(kq \wedge V(0))$ is $b^2_{20}[(0,0)]$. 

By the previous two paragraphs, we see that the class $b^2_{20}[(0,0)] + h^3_0h_3[(1,0)]$ detects $8\sigma[(1,0)] \in \pi_{**}(V(0))$ and detects $\beta[(1,0)] \in \pi_{**}(V(0) \wedge kq)$. Therefore we have $h^{-1}(\beta[(0,0)] = 8\sigma[(1,0)].$ This completes the base case of the induction.

Now suppose that we have shown that 
$$h^{-1}(\beta^i[(0,0)]) = P^{i-1}_v(8\sigma)[(1,0)]$$
for all $i < n$. We can reformulate the induction hypothesis using the motivic Adams spectral sequence as follows: for $i < n$, the class in $\pi_{**}(V(0))$ detected by
$$b^{2i}_{20}[(0,0)] + P^{i-1}_v(h^3_0h_3)[(1,0)]$$
maps under $h$ to the class in $\pi_{**}(kq \wedge V(0))$ detected by $b^{2i}_{20}[(0,0)]$.

To see that this reformulation implies the original induction hypothesis, we need to show that $b^{2i}_{20}[(0,0)]$ detects zero in $\pi_{**}(V(0))$. Using the algebraic squaring operations described in ~\cite[Section 5]{DI10}, we can produce motivic May differentials
$$d_{2^k}(b^{2^k}_{20}) = d_{2^k}(Sq^{2^k}(b^{2^{k-1}}_{20})) = h^{2^k}_0 h_{1+k}$$
with $k \geq 2$. By the Leibniz rule, we can obtain nontrivial May differentials on $b^{2i}_{20}$ for all $i \geq 1$. These elements could support shorter differentials, but in any case we have shown that $b^{2i}_{20}([0,0])$ detects zero in $\pi_{**}(V(0))$. Therefore the first class above detects $P^{i-1}_v(8\sigma)[(1,0)]$ in $\pi_{**}(V(0))$. Since $b^{2i}_{20}[(0,0)]$ detects $\beta^i[(0,0)]$ in $\pi_{**}(kq \wedge M(2))$, we have shown that the reformulation implies the original induction hypothesis.

To complete the induction, consider the diagram
\[
\begin{tikzcd}[row sep={30,between origins}, column sep={75,between origins}]
      & Ext^{***}_A(V(0)) \ar{rr}{P_v}\ar{dl} & & Ext_A^{*+4,*+8,*+4}(V(0)) \ar{dl} \\
   Ext_A^{***}(kq\wedge V(0)) \arrow[crossing over]{rr}{\cdot b^2_{20}}  & & Ext_A^{*+4,*+8,*+4}(kq\wedge V(0)) \\
      &Ext_{A_{cl}}^{**} (V(0) ) \ar{uu} \ar{rr}{P_v} \ar{dl} & &  Ext_{A_{cl}}^{*+4,*+8}(V(0) ) \ar{dl} \ar{uu} \\
    Ext_{A_{cl}}^{**}(ko \wedge V(0)) \ar{rr}{\cdot b^2_{20}} \ar{uu} && Ext_{A_{cl}}^{*+4,*+8} (ko \wedge V(0)). \ar{uu}
\end{tikzcd}
\]
The vertical maps are defined by sending a class with Atiyah-Hirzebruch-May name $h_{ij}[(m)]$ to the element $h_{ij}[(m,0)]$, and similarly for elements detected by products.

In the bottom face of the diagram, consider the class $b^{2n-2}_{20}[(0)] + P^{n-2}_v(h^3_0h_3)[(1)] \in Ext_{A_{cl}}^{**}(V(0))$ which detects $P^{n-2}_v(8\sigma)[(1)]$. Its image in $Ext_{A_{cl}}^{**}(ko \wedge V(0))$ is $b^{2n-2}_{20}[(0)]$ which detects $\beta^{n-1}[(0)]$, its image in $Ext_{A_{cl}}^{*+4*+8}(V(0))$ is $b^{2n}_{20}[(0)] + P^{n-1}_v(h^3_0h_3)[(1)]$ which detects $P^{n-1}_v(8\sigma)[(1)]$, and its image in $Ext_{A_{cl}}^{*+4,*+8}(ko \wedge V(0))$ is $b^{2n}_{20}[0]$ which detects $\beta^n[(0)]$. Indeed, since the Adams periodicity operator $P_v$ is well-defined with no indeterminacy for $P^j_v(h^3_0 h_3)$ for all $j \geq 0$, we can precisely determine which elements in homotopy are detected by these Massey products by sparseness in the image of $J$ ~\cite{Ada66}. 

Now consider the element $P^{n-2}_v(h^3_0 h_3)[(1,0)] \in Ext_A^{***}(V(0))$ in the top face of the cube. By the induction hypothesis, its image in $Ext_A^{***}(kq \wedge V(0))$ is $b_{20}^{2(n-1)}[(0,0)]$. Its inverse image under the vertical map is $P^{n-2}_v(h^3_0 h_3)[(1)]$. We can calculate the images of $P^{n-2}_v(h^3_0 h_3)[(1,0)]$ in the top face using the vertical maps and the previous paragraph. In particular, we conclude that $P^{n-1}_v(h^3_0 h_3)[(1,0)] \in Ext_A^{*+4,*+8,*+4}(V(0))$ maps to $b_{20}^{2n}[(0,0)] \in Ext_A^{*+4,*+8,*+4}(kq \wedge V(0))$. This completes the induction step, so we have proven the values in the table for $x= \beta^k[(0,0)]$. 

The computation of $h^{-1}(\beta^k\alpha[(0,0)])$ is completely analogous. Starting with the differential
$$d_2(h_0b_{20}) = h^3_0h_2$$
which produces the relation $2 \cdot (\eta^3)$, we see that the class
$$h_0b_{20}[(0,0)] + h^3_1[(1,0)]$$
detects $\eta^3[(1,0)] \in \pi_{**}(V(0))$ and detects $\alpha \in \pi_{**}(kq \wedge V(0))$, which shows that $h^{-1}(\alpha[(0,0)]) = \eta^3[(1,0)]$. The same argument using Massey products, algebraic squaring operations, and the commutative cubical diagram completes the induction. 

We now apply the obvious analog of Proposition ~\ref{lift} for the motivic Mahowald invariant based on a finite complex to obtain the following:
\begin{center}
\begin{tabular}{ c c }
$x$ & $M(x;V(0))$\\
$2^{4k}$ & $P^{k-1}_v(8\sigma)[(1,0)]$ \\
 $2^{4k+3}$ & $P^k_v(\eta^3)[(1,0)]$
\end{tabular}
\end{center}

\item Since we know which cell of $\underline{L}^\infty_{-\infty}$ the cosets $M(2^i;V(0))$ were detected on, we obtain the desired refinement to $M(2^i)$.
\end{enumerate}
\end{proof}

\subsection{Motivic Mahowald invariant of $\eta^i$} 

We now compute $M(\eta^i)$. As in the previous subsection, we will state our result in terms of iterarted Toda brackets. We begin by verifying the nontriviality of these Toda brackets; in particular, we show that they detect certain infinite $w^4_1$-periodic families constructed by Andrews in ~\cite{And14}. We begin by recalling some essential information about these families.

By ~\cite[Theorem 3.4]{And14}, the finite complex $C\eta$ admits a non-nilpotent $w^4_1$-self map 
$$\Sigma^{20,12} C\eta \to C\eta.$$
Using the fact that $\nu, \nu^2,\nu^3,i \in \pi_{**}(S^{0,0})$ lift to classes in $\pi_{**}(C\eta)$, one obtains infinite families of maps $P^n(\nu), P^n(\nu^2), P^n(\nu^3), P^n(i) \in \pi_{**}(S^{0,0})$ by composing with
$$C\eta \overset{w^{4n}_1}{\longrightarrow} \Sigma^{-20n,-12n} C\eta \overset{\text{pinch}}{\longrightarrow} S^{2-20n,1-12n}.$$
By ~\cite[Theorem 3.12]{And14}, these compositions are nontrivial for all $n \geq 0$. We note that $P(i) = \eta^2 \eta_4$.  

To prove nontriviality of these composites, Andrews uses detection maps from the $E_2$-page of certain motivic Adams-Novikov spectral sequences to the $E_2$-page of certain classical Adams spectral sequences ~\cite[Definition 2.8]{And14}. For example, he considers the map
$$d : Ext^{***}_{BPGL_{**}BPGL}(BPGL_{**},BPGL_{**}) \to Ext^{**}_{A_{cl}}(\f_2,\f_2)$$
defined by $\tau \mapsto 0$, $v_n \mapsto 0$, and $t_n \mapsto \zeta_n$. This is a graded map if one defines the degree of an element in the source to be its motivic weight. He uses the detection maps to infer the nontriviality of the classes $P^n(x)$ in $Ext^{***}_{BPGL_{**}BPGL}(BPGL_{**},BPGL_{**})$ by showing that their images are the infinite families constructed by Adams in ~\cite{Ada66b}. 

\begin{lem}
The iterated Massey products $g^i(x)$ for $x = \eta^2\eta_4, \nu,\nu^2,\nu^3$ are permanent cycles in the $C\tau$-linear $\overline{H}$-based Adams spectral sequence. Moreover, they detect the $w^4_1$-periodic families $P^i(x)$ discussed above.
\end{lem}

\begin{proof}
By Corollary ~\ref{mw1per} and ~\cite{Ada66b}, these iterated Massey products are nontrivial with zero indeterminacy. We claim that they are permanent cycles which detect the infinite $w^4_1$-periodic families constructed by Andrews in ~\cite{And14}. Composing Andrews' detection map above with the doubling homomorphism gives a map
$$\cd \circ d : Ext^{***}_{BPGL_{**}BPGL}(BPGL_{**},BPGL_{**}) \to Ext^{***}_{\overline{A}}(\f_2,\f_2).$$

By the computations of the images of the detection map in ~\cite[Section 4]{And14} along with Remark ~\ref{bop}, we see that the images of the classes $\alpha_{4/4}^i$ for $1 \leq i \leq 3$ and $1$ in the motivic Adams-Novikov spectral sequence under under $\cd \circ d$ are $h^i_2$ and $1$ in the $C\tau$-linear $\overline{H}$-based Adams spectral sequence. Moreover, we can relate the class $x \in Ext^{4,24,12}_{BPGL_{**}BPGL}(BPGL_{**}(C\eta))$ constructed in ~\cite[Proposition 3.3]{And14} which maps to $\alpha^2_1 \beta_{4/3}$ under the collapse map to a class in the $C\tau$-linear $\overline{H}$-based Adams spectral sequence. The class $\alpha^2_1 \beta_{4/3}$ maps to $h^3_1 h_4$ under the composition of the detection map with the doubling map, so the same argument as in the proof of ~\cite[Proposition 3.3]{And14} defines an element $x' \in Ext^{4,24,12}_{\overline{A}}(\overline{H}^{**}(C\eta))$ which maps to $h^3_1 h_4$ under the analogous composition. Further, it follows from the $C\tau$-linear $\overline{H}$-based May differential $d_4(b^2_{21}) = h^4_1 h_4$ that $x'$ has the Atiyah-Hirzebruch-May name $b^2_{21}[(0,0)]$. This $d_4$-differential follows from the same argument as the motivic May differential of the same name.

Therefore multiplication by $b^2_{21}$ detects $w^4_1$ in the $C\tau$-linear $\overline{H}$-based Adams spectral sequence converging to $\pi_{**}(C\eta)$. Since we have already shown that the Massey products $g(x)$ are nontrivial with zero indeterminacy, we have
$$g(x) = \langle h^4_1, h_4, x \rangle = b^2_{21} x.$$
Therefore the periodicity operators $g(-)$ defined above and the periodicity operator $P(-)$ defined by Andrews both detect $w^4_1 x$. Since Andrews infinite families are permanent cycles, so are the infinite families $g^i(x)$ for $x$ as stated in the lemma. 
\end{proof}

\begin{thm}
Let $i \geq 1$. The motivic Mahowald invariant of $\eta^i$ is given by
\[ M(\eta^i) \ni \begin{cases}
g^{\lfloor i/4 \rfloor}(h_1^3h_4) \quad & i \equiv 0 \mod 4, \\
g^{\lfloor i/4 \rfloor}(h_2) & i \equiv 1 \mod 4,\\
g^{\lfloor i/4 \rfloor}(h_2^2) & i \equiv 2 \mod 4,\\
g^{\lfloor i/4 \rfloor}(h_2^3) & i \equiv 3 \mod 4.
\end{cases}
\]
Here we are denoting nontrivial Toda brackets in $\pi_{**}(S^{0,0})$ by the Massey products which detect them.
\end{thm}

\begin{proof}
$(i \equiv 1,2\mod 4)$ The elements $\beta^k\nu, \beta^k\nu^2 \in \pi_{**}(wko)$ are in the Hurewicz image for $wko$. Their inverse images are $g^k(\nu), g^k(\nu^2) \in \pi_{**}(S^{0,0})$, so this case is clear by Proposition ~\ref{lift} applied to Proposition ~\ref{wkoeta}. 

$(i \equiv 0,3 \mod4)$. The elements $\alpha, \beta \in \pi_{**}(wko)$ are not in the Hurewicz image, so we cannot immediately employ Proposition ~\ref{lift}. Instead, we must pass through the motivic Mahowald invariant based on the $4$-cell complex $C := \underline{L}^1_{-2}$. We compute approximations in the following order:
$$M_{wko}(\eta^i) \rightsquigarrow M_{wko}(\eta^i;C) \rightsquigarrow M_{V(0)}(\eta^i;C) \rightsquigarrow M_{V(0)}(\eta^i) \rightsquigarrow M(\eta^i).$$

\begin{enumerate}
\item By Proposition ~\ref{wkoeta}, we have computed $M_{wko}(\eta^i)$.
\item By the proof of Proposition ~\ref{wkoeta}, we know on which cells of $\underline{L}^\infty_{-\infty}$ the elements $\eta^{4k}$ and $\eta^{4k+3}$ are first detected. The topological degree of the cell determines where in the filtration of $\underline{L}^\infty_{-\infty}$ by $C$ the coset $M_{wko}(\eta^i;C)$ is detected. We obtain the following table of values
\begin{center}
\begin{tabular}{ c c c c }
$x$ & s & N & $M_{wko}(x;C)$\\
$\eta^{4k}$ & $-16k$ & $1+8k$   & $\beta^{k}[(1,1)]$ \\
$\eta^{4k+3}$ & $-8-16k$ & $5+4k$ & $\beta^k \alpha[(1,1)]$ \\
\end{tabular}
\end{center}
To justify that $\beta^k \alpha^\epsilon[(1,1)]$, $\epsilon \in \{0,1\}$, are the correct Atiyah-Hirzebruch-May names for the classes above, it suffices to note that $\beta^k \alpha^\epsilon \in wko_{**}$ is not in the image of multiplication by $\eta$ for any $k, \epsilon$. 

\item We now use the canonical map $h: V(0) \to wko$ to compute $M_{V(0)}(\eta^i;C)$ from $M_{wko}(\eta^i;C)$. We need to assign Atiyah-Hirzebruch-May names to elements $\alpha \in \pi_{**}(V(0) \wedge C)$. Our notation will be that $\alpha$ is detected by $x[(m_1,n_1),(m_2,n_2)]$ where $(m_2,n_2)$ comes from the Atiyah-Hirzebruch-May name for $\alpha \in \pi_{**}(C)$ and $(m_1,n_1)$ comes from the Atiyah-Hirzebruch-May name for $\alpha \in \pi_{**}(V(0) \wedge C)$. 

We will prove the following by induction on $k$:
\begin{center}
\begin{tabular}{ c c  }
$x$ & $h^{-1}(x)$\\
$\beta^{k}[(1,1)]$ & $g^{k-1}(\eta^2\eta_4)[(0,0),(3,2)]$\\
$\beta^k \alpha[(1,1)]$ & $g^k(\nu^3)[(0,0),(3,2)].$ 
\end{tabular}
\end{center}
We start with the case $x=\beta^k[(1,1)]$. We can obtain the Atiyah-Hirzebruch-May name for an element in $\pi_{**}(V(0) \wedge C)$ arising from $\eta^2 \eta_4 \in \pi_{**}(S^{0,0})$ using cofiber sequences as follows. As in the proof of Theorem ~\ref{mm2i}, the maps $f_*,j_*,\delta$ are the maps in the long exact sequence in homotopy groups associated to a cofibration. First, recall that $C\eta$ is defined by the cofiber sequence
$$S^{1,1} \overset{\eta}{\to} S^{0,0} \to C\eta.$$
By ~\cite[Proposition 5.5]{DI10}, there is a motivic May differential $d_4(b^2_{21}) = h^4_1 h_4$. The same proof implies that there is a $C\tau$-linear $\overline{H}$-based May differential $d_4(b^2_{21}) = h^4_1 h_4$. In the Adams spectral sequence for $\pi_{**}(S^{0,0})$, this differential produces the relation $\eta \cdot (\eta^2 \eta_4) = 0$. Then the Atiyah-Hirzebruch-May name for $\delta^{-1}(\eta^2 \eta_4) \in \pi_{**}(C\eta)$ is $h^3_1h_4[(2,1)]$. Now, $C$ is the defined by the cofiber sequence
$$\Sigma^{-1,0} C\eta \overset{f}{\to} \Sigma^{1,1} C\eta \to C$$
where $f$ is the composition of the collapse onto the top cell and multiplication by $2$. Therefore $C$ has cells in dimensions $(0,0)$ and $(2,1)$ coming from $\Sigma^{-1,0} C\eta$ and cells in dimensions $(1,1)$ and $(3,2)$ coming from $\Sigma^{1,1} C\eta$. Although $h^2_0 h_2h_4 = \tau h^3_1 h_4$ in the motivic Adams spectral sequence, the class $\eta^2 \eta_4 \in \pi_{**}(S^{0,0})$ detected by $h^3_1 h_4$ is not in the image of $\cdot 2$ ~\cite{Isa14b}. Therefore we see that $\eta^2 \eta_4[(3,2)] \in \pi_{**}(\Sigma^{1,1}C\eta)$ is not in the image of $f_*$ and so we obtain the Atiyah-Hirzebruch-May name $h^3_1h_4[(3,2)]$ for the element $j_*(\eta^2\eta_4[(3,2)]) \in \pi_{**}(C)$. Finally, the complex $V(0) \wedge C$ is defined by the cofiber sequence
$$C \overset{\cdot 2}{\to} C \to V(0) \wedge C.$$
By the same reasoning as above, we see that the Atiyah-Hirzebruch-May name for the element $j_*(\eta^2\eta_4[(3,2)]) \in \pi_{**}(V(0) \wedge C)$ is $h^3_1h_4[(0,0),(3,2)]$. Note that the above May differential implies that $b^2_{21}$ detects zero in $\pi_{**}(S^{0,0})$, and so it also detects zero in $\pi_{**}(V(0) \wedge C)$. 

On the other hand, $b^2_{21}$ detects $\beta \in \pi_{**}(wko)$ by the $C\tau$-linear version of the classical argument that $b^2_{20}$ detects $\beta \in \pi_*(ko)$. Since $b^2_{21}$ supports multiplication by $\eta$, does not support multiplication by $2$, and is not in the image of multiplication by $\eta$ or $2$, the same series of cofiber sequences as above shows that there is a class in $\pi_{**}(wko \wedge V(0) \wedge C)$ with Atiyah-Hirzebruch-May name $b^2_{21}[(0,0),[(1,1)]$. Since $h_4$ does not appear in the $C\tau$-linear $\overline{H}$-based May spectral sequence converging to $Ext_{\overline{A}(1)}(\f_2,\f_2)$, we see that $h^3_1 h_4$ detects zero in $\pi_{**}(kq)$. In particular, this implies that the class above detects zero in $\pi_{**}(wko \wedge V(0) \wedge C)$. 

By the previous two paragraphs, we see that the class 
$$b^2_{21}[(0,0),(1,1)] + h^3_1h_4[(0,0),(3,2)]$$
detects $\eta^2 \eta_4[(0,0),(3,2)]$ in $\pi_{**}(V(0) \wedge C)$ and detects $\beta[(0,0),(1,1)]$ in $\pi_{**}(wko \wedge V(0) \wedge C)$. Therefore we have $h^{-1}(\beta[(0,0),(1,1)]) = \eta^2 \eta_4[(0,0),(3,2)]$. This completes the base case of the induction.

Now suppose that we have shown that 
$$h^{-1}(\beta^i([1,1])) = P^{i-1}(\eta^2\eta_4)[(0,0),(3,2)]$$
for all $i < n$. We can reformulate the induction hypothesis using the $C\tau$-linear $\overline{H}$-based Adams spectral sequence as follows: for $i < n$, the class in $\pi_{**}(V(0) \wedge C)$ detected by
$$b^{2i}_{21}[(0,0),(1,1)] + g^{i-1}(h^3_1h_4)[(0,0),(3,2)]$$
maps under $h$ to the class in $\pi_{**}(wko \wedge V(0) \wedge C)$ detected by $b^{2i}_{21}[(0,0),(1,1)]$.

To see that this reformulation implies the original induction hypothesis, we need to show that $b^{2i}_{21}[(0,0),(1,1)]$ detects zero in $\pi_{**}(V(0) \wedge C)$ for all $i \geq 1$. Using algebraic squaring operations, we can produce May differentials
$$d_{2^k}(b^{2^k}_{21}) = d_{2^k}(Sq^{2^{k}}(b^{2^{k-1}}_{21}) )= h^{2^k}_1 h_{2+k}$$
with $k \geq 2$ as described in ~\cite[Remark 5.7]{DI10}. By the Leibniz rule, we can obtain nontrivial May differentials on $b^{2i}_{21}$ for all $i \geq 1$. These elements could support shorter differentials, but in any case we have shown that $b^{2i}_{21}$ detects zero in $\pi_{**}(S^{0,0})$ and so $b^{2i}_{21}[(0,0),(1,1)]$ detects zero in $\pi_{**}(V(0) \wedge C)$. Therefore the sum above detects $g^{i-1}(\eta^2 \eta_4)[(0,0),(3,2)]$ in $\pi_{**}(C \wedge V(0))$. Since $b^{2i}_{21}[(0,0),(3,2)]$ detects $\beta^i[(0,0),(1,1)]$ in $\pi_{**}(wko \wedge V(0) \wedge C)$, we have shown that the reformulation implies the original induction hypothesis.

To complete the induction, consider the diagram
\[
\begin{tikzcd}[row sep={30,between origins}, column sep={85,between origins}]
      & Ext_A^{***}(V(0) \wedge C) \ar{rr}{g}\ar{dl} & & Ext_A^{*+4,*+20,*+12}(V(0) \wedge C) \ar{dl} \\
   Ext_A^{***}(wko \wedge V(0) \wedge C) \arrow[crossing over]{rr}{\cdot b^2_{21}} & & Ext_A^{*+4,*+20,*+12}(wko\wedge V(0) \wedge C) \\
      &Ext_{A_{cl}}^{**} (V(0) \wedge C) \ar{uu} \ar{rr}{P_v} \ar{dl} & &  Ext_{A_{cl}}^{*+4,*+8}(V(0) \wedge C) \ar{dl} \ar{uu} \\
    Ext_{A_{cl}}^{**}(ko \wedge V(0) \wedge C) \ar{rr}{\cdot b^2_{20}} \ar{uu} && Ext_{A_{cl}}^{*+4,*+8} (ko \wedge V(0) \wedge C). \ar{uu}
\end{tikzcd}
\]
The vertical maps are induced by sending the class $h_{ij}[(m)(n)]$ to the class $h_{i,j+1}[(m,0),(n,\lfloor (n+1)/2 \rfloor)]$, and similarly for classes detected by products and Massey products. 

In the bottom face of the diagram, consider the element $b^{2n-2}_{20}[(0),(1)] + P^{n-2}_v(h^3_0h_3)[(0),(3)]$ which detects $P^{n-2}_v(8\sigma)[(0),(3)] \in \pi_{*}(V(0) \wedge C)$. Its image in $\Ext_{A_{cl}}^{**}(ko \wedge V(0) \wedge C)$ is $b^{2n-2}_{20}[(0),(1)]$ which detects $b_{20}^{2(n-1)}[(0),(1)]$, its image in $Ext_{A_{cl}}^{*+4,*+8}(V(0) \wedge C)$ is $b^{2n}_{20}[(0),(1)] + P^{n-1}_v(h^3_0h_3)[(0),(3)]$ which detects $P^{n-1}_v(8\sigma)[(0),(1)]$, and its image in $Ext_{A_{cl}}^{*+4,*+8}(ko \wedge V(0) \wedge C)$ is $b^{2n}_{20}[(0),(1)]$ which detects $\beta^n[(0),(1)]$. This follows from the same argument as in the proof of Theorem ~\ref{mm2i}. 

Now consider the element $g^{n-2}(h_1^3 h_4)[(0,0),(3,2)] \in Ext_{\overline{A}}^{***}(V(0) \wedge C)$ in the top face of the cube. By the induction hypothesis, its image in $Ext_{\overline{A}}^{***}(wko \wedge V(0) \wedge C)$ is the class $b_{21}^{2(n-1)}[(0,0),(1,1)]$. Its inverse image under the vertical map is $P^{n-2}_v(h^3_0 h_3)[(0),(1)]$. We can calculate the images of $g^{n-2}(h_1^3 h_4)[(0,0),(3,2)]$ in the top face using the vertical maps and the previous paragraph. In particular, we conclude that $g^{n-1}(h^3_1 h_4)[(0,0),(3,2)] \in Ext_{\overline{A}}^{*+4,*+20,*+12}(V(0) \wedge C)$ maps to $b_{21}^{2n}[(0,0),(1,1)] \in Ext_{\overline{A}}^{*+4,*+20,*+12}(wko \wedge V(0) \wedge C)$. This completes the induction step, so we have proven the values in the table for $x= \beta^k[(1,1)]$. 

The computation of $h^{-1}(\beta^k\alpha[(1,1)])$ is completely analogous. Starting with the May differential
$$d_4(b_{20}h^4_1) = h_1h^3_2$$
which produces the relation $\eta \cdot (\nu^3)$ in $\pi_{**}(S^{0,0})$, we see that the class
$$b_{20}h^4_1[(0,0),(1,1)] + h^3_2[(0,0),(3,2)]$$
detects $\nu^3[(0,0),(3,2)] \in \pi_{**}(V(0) \wedge C)$ and detects $\alpha \in \pi_{**}(wko \wedge V(0) \wedge C)$, which shows that $h^{-1}(\alpha[(1,1)]) = \nu^3[(0,0),(3,2)]$. The same argument using Massey products, algebraic squaring operations, and the commutative cubical diagram completes the induction. 

\item Since we know which cell of $\underline{L}^\infty_{-\infty}$ the coset $M_{V(0)}(\eta^i;C)$ is detected on, we obtain the following refinement of the previous computation: 
\begin{center}
\begin{tabular}{ c c  }
$x$ & $M_{V(0)}(x)$\\
$\eta^{4k+3}$ & $g^{k-1}(\nu^3)[(0,0)]$\\
$\eta^{4k+4}$ & $g^{k-1}(\eta^2 \eta_4)[(0,0)].$ 
\end{tabular}
\end{center}
\item Proposition ~\ref{lift} applied to the inclusion of the bottom cell $S^{0,0} \hookrightarrow V(0)$ completes the computation. 
\end{enumerate}
\end{proof}

\bibliographystyle{plain}
\bibliography{master}

\end{document}